\documentclass[11pt, twoside]{amsart}
\calclayout

\title[Semisimplification of contragredient Lie algebras]{Semisimplification of contragredient Lie algebras}

\author{Iv\'an Angiono}
\address{FaMAF-CIEM (CONICET), Universidad Nacional de C\'ordoba, Medina Allende s/n, Ciudad Universitaria, 5000 C\'ordoba, Rep\'ublica Argentina}
\email{ivan.angiono@unc.edu.ar}
\urladdr{}

\author{Julia Plavnik}
\address{Department of Mathematics, Indiana University, Bloomington, IN 47405, USA.}
\email{jplavnik@iu.edu}
\urladdr{}

\author{Guillermo Sanmarco}
\address{Department of Mathematics, University of Washington, 
Seattle, WA 98195, USA.}
\email{sanmarco@uw.edu}
\urladdr{}

\usepackage{mathabx}
\usepackage{import}
\usepackage{amsmath}
\usepackage{amsfonts}
\usepackage{amsthm}
\usepackage{amssymb,bbm}
\usepackage{mathrsfs}  
\usepackage{dsfont}
\usepackage[alphabetic, initials]{amsrefs}
\usepackage[english]{babel}
\usepackage{url}
\usepackage{fancyhdr}
\usepackage{graphicx}
\usepackage{verbatim}
\usepackage[inline]{enumitem}
\usepackage{mathtools}
\usepackage[
colorlinks=true,
linkcolor=black, 
anchorcolor=black,
citecolor=black,
urlcolor=black, 
]{hyperref}
\usepackage{microtype} 
\usepackage[dvipsnames]{xcolor}

\usepackage{tikz-cd}
\usepackage{todonotes}

\usepackage{cleveref}
\usepackage{adjustbox}
\usepackage[T1]{fontenc}
\usepackage{lmodern}
\usepackage{multicol}
\usepackage[all]{xy}
\SelectTips{cm}{}
\usepackage{tikz}
\usetikzlibrary{decorations.markings}
\usetikzlibrary{decorations.pathreplacing}
\usetikzlibrary{arrows,shapes,positioning}

\newcommand{\ad}{\operatorname{ad}}

\newcommand{\GL}{\operatorname{GL}}
\newcommand{\gr}{\operatorname{gr}}
\newcommand{\Hom}{\mathsf{Hom}}

\newcommand{\one}{\mathds{1}}
\newcommand{\re}{\operatorname{re}}
\newcommand{\rk}{\operatorname{rank}}

\newcommand{\id}{\operatorname{id}}

\newcommand{\sd}{\operatorname{sdim}}
\newcommand{\ond}{\operatorname{o,nd}}

\newcommand{\balp}{\boldsymbol\alpha}
\newcommand{\sfS}{\mathsf{S}}

\newcommand{\height}{\mathsf{ht}}

\newcommand\restr[2]{{\left.\kern-\nulldelimiterspace #1 \vphantom{\big|} \right|_{#2}}} 

\newcommand{\Rep}{\mathsf{Rep}}

\renewcommand{\Vec}{\mathsf{Vec}_{\Bbbk}}
\newcommand{\sVec}{\mathsf{sVec}_\Bbbk}
\newcommand{\Ver}{\mathsf{Ver}}
\newcommand{\Verp}{\mathsf{Ver}_p}
\newcommand{\ot}{\otimes}

\newcommand{\OLie}{\mathsf{OLie}}

\newcommand{\indcat}[1]{#1^{\op{ind}}}

\providecommand{\op}[1]{\operatorname{#1}}
\newcommand{\ov}[1]{\overline{#1}}

\newcommand{\bF}{\mathbb{F}}
\newcommand{\bG}{\mathbb{G}}
\newcommand{\bI}{\mathbb{I}}
\newcommand{\bJ}{\mathbb{J}}

\newcommand{\bR}{\mathbb{R}}

\newcommand{\bZ}{\mathbb{Z}}
\newcommand{\bN}{\mathbb{N}}
\newcommand{\bk}{\Bbbk}

\newcommand{\cC}{\mathcal{C}}
\newcommand{\cD}{\mathcal{D}}

\newcommand{\cG}{\mathcal{G}}

\newcommand{\cI}{\mathcal{I}}

\newcommand{\cM}{\mathcal{M}}

\newcommand{\cR}{\mathcal{R}}

\newcommand{\cW}{\mathcal{W}}
\newcommand{\cX}{\mathcal{X}}

\newcommand{\fc}{\mathfrak{c}}
\newcommand{\fg}{\mathfrak{g}}
\newcommand{\fh}{\mathfrak{h}}
\newcommand{\fsl}{\mathfrak{sl}}

\newcommand{\fn}{\mathfrak{n}}

\newcommand{\frr}{\mathfrak{r}}
\newcommand{\fH}{\mathfrak{H}}

\newcommand{\tfg}{\widetilde{\mathfrak{g}}}
\newcommand{\tfn}{\widetilde{\mathfrak{n}}}

\newcommand{\fb}{\widehat{f}}

\newcommand{\ttb}{\mathtt{b}}

\newcommand{\ttL}{\mathtt{L}}

\newcommand{\bp}{\mathbf{p}}
\newcommand{\bq}{\mathbf{q}}
\newcommand{\brown}{\mathfrak{br}}
\newcommand{\sbrown}{\mathfrak{brj}}
\newcommand{\osp}{\mathfrak{osp}}
\newcommand{\supersl}[2]{{\mathfrak{sl} }(#1|#2)}

\numberwithin{equation}{section}

\newtheorem*{rep@theorem}{\rep@title}
\newcommand{\newreptheorem}[2]{%
\newenvironment{rep#1}[1]{%
 \def\rep@title{#2 \ref{##1}}%
 \begin{rep@theorem}}%
 {\end{rep@theorem}}}
\makeatother

\newtheorem{theorem}{Theorem}[section]
\newtheorem{proposition}[theorem]{Proposition}
\newreptheorem{proposition}{Proposition}
\newtheorem{corollary}[theorem]{Corollary}
\newreptheorem{corollary}{Corollary}
\newtheorem{lemma}[theorem]{Lemma}

\newtheorem{theorem*}{Theorem}
\newreptheorem{theorem}{Theorem}
\newtheorem{claim}{Claim}
\usepackage{etoolbox}
\AtEndEnvironment{proof}{\setcounter{claim}{0}}

\theoremstyle{definition}
\newtheorem{definition}[theorem]{Definition}
\newtheorem{notation}[theorem]{Notation}

\newtheorem{example}[theorem]{Example}
\newtheorem{remark}[theorem]{Remark}

\makeatletter            
\let\c@equation\c@theorem  
\makeatother
\numberwithin{equation}{section}

\makeatletter
\@namedef{subjclassname@2020}{%
  \textup{2020} Mathematics Subject Classification}
\makeatother

\subjclass[2020]{}
\keywords{}

\begin{document}

\begin{abstract}
We describe the structure and different features of Lie algebras in the Verlinde category, obtained as semisimplification of contragredient Lie algebras in characteristic $p$ with respect to the adjoint action of a Chevalley generator. In particular, we construct a root system for these algebras that arises as a parabolic restriction of the known root system for the classical Lie algebra.  This gives a lattice grading with simple homogeneous components and a triangular decomposition for the semisimplified Lie algebra. We also obtain a non-degenerate invariant form that behaves well with the lattice grading.
As an application, we exhibit concrete new examples of Lie algebras in the Verlinde category.
\end{abstract}

\maketitle

\tableofcontents


\section{Introduction}

In recent times, symmetric tensor categories have attracted significant attention. 
The work of Deligne \cites{Del90, Del02} has been a foundational stone for this theory, and it is the major influence in the modern development of the subject. By Deligne's results, every pre-Tannakian category (that is, a symmetric tensor category with objects of finite length) of moderate growth over an algebraically closed field $\Bbbk$ of characteristic zero admits a fiber functor to $\sVec$, the category of finite-dimensional super vector spaces. In other words, any such category is equivalent to the category of representations of a (pro)algebraic supergroup.

Over a field $\Bbbk$ of characteristic $p>0$ the situation is completely different. For example, the Verlinde category $\Ver_p$, introduced in \cites{GelKaz,GeoMat} as the semisimplification of $\Rep(\bZ_p)$, does not admit a fiber functor to $\sVec$.
In \cite{Ost2}, Ostrik initiated the quest of providing analogs to Deligne's theorem in positive characteristic. Ostrik's main result states that, in characteristic $p$, any pre-Tannakian category which is fusion (thus, of moderate growth) admits a fiber functor to $\Ver_p$. This result opens the door to different directions of research concerning pre-Tannakian categories of moderate growth, such as the notion of incompressible tensor categories \cites{BEO,CEO-incomp}, necessary and sufficient conditions for such a category to fiber over $\Verp$ \cite{CEO}, and the study of affine group schemes in $\Ver_p$ \cites{Ven-GLVerp,Cou-comm}. This work focuses on the latter direction, as we study Lie algebras in $\Ver_p$.

\medbreak
Finite-dimensional Lie superalgebras over fields of characteristic zero were classified by Kac in the seventies \cite{Kac-super}. A distinguished class in the classification is that of contragredient Lie algebras. These are Lie algebras defined from a matrix by generators and relations, slightly generalizing the construction of Kac-Moody algebras but still preserving many desired properties, such as a  triangular decomposition \cite{Kac-book}.
This family includes analogs of classical Lie algebras as well as special, orthogonal, and symplectic Lie superalgebras, and some exceptions in low rank. Once again the situation is a bit more complicated for characteristic $p>0$. In this context, the classification of finite-dimensional contragredient Lie superalgebras was achieved in \cite{BGL} and contains several exceptions when $p=3,5$, partially constructed \emph{by hand}. 

A possible explanation for the existence of some of the exceptions comes from the \emph{super magic square} \cite{Cunha-Elduque}, which generalizes the Freudenthal magic square for Lie algebras. More recently a different approach was given by Kannan \cite{Kan}, who constructed these exceptional Lie superalgebras using the modern theory of symmetric tensor categories in positive characteristic. 
More precisely, Kannan showed that all the exotic examples can be obtained via the semisimplification process. The starting point is the realization of $\Ver_p$ as the semisimplication of the representation category for the commutative Hopf algebra $\bk[t]/(t^p)$, where $t$ is primitive. Thus, there is a semisimplification functor $\Rep(\bk[t]/(t^p))\to \Verp$ which is symmetric. In this way, one can obtain Lie algebras in $\Verp$ by applying that semisimplification functor to Lie algebras in $\Rep(\bk[t]/(t^p))$. Now, a Lie algebra in the latter category is just a pair $(\fg,\partial)$, where $\fg$ is a finite-dimensional Lie algebra and $\partial$ is a derivation of $\fg$ such that $\partial^p=0$. In characteristic $p=3$, we have $\Ver_3=\sVec$ so the process always gives a Lie superalgebra. For $p=5$ (and higher), we have a factorization $\Ver_5=\sVec\boxtimes\Ver_5^+$ as symmetric categories, thus the semisimplification procedure still gives Lie superalgebras after projecting to $\sVec$. Applying this construction to particular pairs $(\fg,\partial)$, where $\fg$ is a classical Lie algebra and $\partial$ is a certain inner derivation, Kannan recovered all exceptional examples, which only exist in characteristics $3$ and $5$.

\medbreak

However, the semisimplification process for general $p$ is, at least in principle, a possible source of new examples of Lie algebras in $\Ver_p$, not necessarily supported in $ \sVec$, which are much desired. 
That idea is the starting point of this work. We study the structure of Lie algebras in $\Ver_p$ obtained as the semisimplification of a pair $(\fg, \partial)$, where $\fg$ is a contragredient Lie algebra and $\partial$ is an inner derivation associated to a homogeneous element with respect to the lattice grading. This homogeneity assumption is mild enough to still produce new examples of Lie algebras in $\Verp$, yet simultaneously manageable to assure nice features on these algebras. Most importantly, we obtain a grading by a suitable free abelian group which resembles that of Lie superalgebras.
Indeed, contragredient Lie superalgebras admit a grading which leads to (generalized) root systems in the sense of \cites{Heckenberger-Yamane-root-system,HS}, see e.g. \cite{AA}. In our case, we induce a grading on the semisimplified Lie algebra in $\Ver_p$ coming from that on the original Lie algebra $\fg$: because of the choice of a homogeneous element, the \emph{root system} on the Lie algebra in $\Ver_p$ is the parabolic restriction (in the sense of \cite{Cuntz-Lentner}) of the original one. One may wonder which other Lie algebras in $\Ver_p$ have root systems. This is the content of a forthcoming paper.

\medbreak
The organization of the paper is the following. In \S \ref{sec:contragredient} we recall the construction of contragredient Lie superalgebras since they appear along the work in two ways: the particular case of Lie algebras is a source of examples to being semisimplified, and Lie superalgebras, mainly in characteristic $3$, are the images of the semisimplification functor. In particular, we recall the notion of root system, some properties, and the parabolic restriction from \cite{Cuntz-Lentner}. Later, in \S \ref{sec:Lie-algs-STC}, we recall the notions of symmetric tensor categories, Lie algebras on these categories, the semisimplification functor, and the Verlinde category $\Ver_p$.
The main results are part of \S \ref{sec:ss-Lie-algebras}. Here we consider a contragredient Lie algebra $\fg(A)$ attached to a matrix $A$ and a homogeneous element $x$. Up to an isomorphism of $\fg(A)$, we may assume that $x=e_i$, a generator of the \emph{positive} part of $\fg(A)$. We then study this case in detail, first by understanding the structure of $\fg(A)$ as a module over $\bk[t]/(t^p)$, and later by studying its image under the semisimplification functor. If $A$ is of rank $\theta$, it is known that $\fg(A)$ is $\bZ^{\theta}$-graded. We show that the semisimplification of $\fg(A)$ is $\bZ^{\theta-1}$-graded, where the grading comes from the projection which annihilates the $i$-th entry and gives consequently a grading coming from the parabolic restriction of the root system of $\fg(A)$. We finish the paper with some explicit examples in low rank.

\section{Contragredient Lie superalgebras}\label{sec:contragredient}

\subsection{Conventions}\label{subsec:conventions}
We denote by $\bN$ the set of positive integers and $\bN_0=\bN\cup \{0\}$. Given $\theta\in\bN$, let $\bI_\theta=\{1, \dots, \theta\}$. If $\theta$ is implicit we just write $\bI=\bI_\theta$; in this case, we make no distinction between $\bZ^\bI$ and $\bZ^\theta$. The canonical basis of $\bZ^\bI$ is $(\alpha_i)_{i \in \bI}$; we use the expression $1^{a_1}\dots \theta^{a_{\theta}}$ to denote
$a_1\alpha_1+\dots+a_{\theta}\alpha_{\theta}\in \mathbb{Z}^{\bI}$. For each $\beta=1^{a_1}\dots \theta^{a_{\theta}}\in\bN_0^{\theta}$, the height of $\beta$ is $\height(\beta)\coloneq a_1+\dots+a_{\theta}\in\bN_0$.

We work over an algebraically closed field $\bk$ of characteristic $p > 0$. 
We denote by $\Vec$ (respectively $\sVec$) the symmetric fusion category of finite dimensional $\bk$-vector spaces, (respectively super vector spaces).


\subsection{Root systems and Weyl groupoids}

We recall here the notions of root systems and Weyl groupoids from \cite{HS}, a generalization of the classical definitions of root systems and Weyl groups for Lie algebras.

\smallbreak
Fix a natural number $\theta$ and a nonempty set $\cX$. A \emph{semi-Cartan graph} of \emph{rank} $\theta$ and a set of \emph{points} $\cX$
is a quadruple $\cG\coloneq\cG(\bI,\cX,(A^X)_{X\in\cX},(r_i)_{i\in\bI})$, where
\begin{itemize}[leftmargin=*]
\item for each $i\in\bI$, $r_i:\cX\to\cX$ is a function such that $r_i^2=\id_{\cX}$;
\item $A^X=(a_{ij}^X)_{i,j\in\bI}\in\bZ^{\theta\times\theta}$ is a generalized Cartan matrix \cite{Kac-book} for all $X\in\cX$;

\noindent and the following identities hold:
\begin{align}\label{eq:semi-Cartan-graph-defn}
a_{ij}^X &= a_{ij}^{r_i(X)}, \qquad \text{for all }X\in\cX, i,j\in\bI.
\end{align}
\end{itemize}
The exchange graph of $\cG$ is a graph with $\cX$ as set of vertices, and an arrow labelled with $i\in\bI$ between $X$ and $r_i(X)$ if $X\ne r_i(X)$, for each $i\in\bI$ and $X\in\cX$. An example of a exchange graph and semi-Cartan graph for $i=2$ and $\cX=\{X_1,X_2,X_3\}$ is the following:
\begin{align*}
& \xymatrix{\underset{X_1}{\circ} \ar@{-}^{1}[r] & \underset{X_2}{\circ} \ar@{-}^{2}[r] & \underset{X_3}{\circ}}, &
& A^{X_1}=\begin{bmatrix} 2 & -2 \\ -2 & 2 \end{bmatrix}, &
& A^{X_2}=\begin{bmatrix} 2 & -2 \\ -1 & 2 \end{bmatrix}, &
& A^{X_3}=\begin{bmatrix} 2 & -4 \\ -1 & 2 \end{bmatrix}.
\end{align*}

\smallbreak

Given a monoid $\cM$, we may consider the small category $\cD(\cX,\cM)$ whose set of objects is $\cX$ and the set of morphisms between any two objects is $\cM$. We use the following notation:
\begin{align*}
\Hom(X,Y)&=\{(Y,f,X) | f\in\cM\} & \text{for each pair } & X,Y\in\cX.
\end{align*}
The composition is then written as follows:
\begin{align*}
(Z,f,Y)\circ (Y,g,X) &= (Z,fg,X), & &\text{for any }X,Y,Z\in\cX, f,g\in\cM.
\end{align*}

The Weyl groupoid $\cW\coloneq\cW(\bI,\cX,(A_x)_{x\in\cX},(r_i)_{i\in\bI})$ of the semi-Cartan graph $\cG$ is defined as the full subcategory of $\cD(\cX,\GL(\bZ^{\theta}))$ generated by
\begin{align*}
\sigma_i^X &\coloneq (r_i(X),s_i^X,X), && i\in\bI, X\in\cX,
\end{align*}
where $s_i^X\in\GL(\bZ^{\theta})$ is given by $s_i^X(\alpha_j)=\alpha_j-a_{ij}^X\,\alpha_i$. Notice that $\sigma_i^{r_i(X)}\sigma_i^X=(X,\id_X,X)$ for all $i\in\bI$ and $X\in\cX$, so $\cW$ is indeed a groupoid. 
\smallbreak

For a semi-Cartan graph $\cG$ as above, we define the set of \emph{real roots} of $\cG$ at $X\in\cX$ as
\begin{align*}
\varDelta^{X,\re} &\coloneq \{w(\alpha_i)\in\bZ^{\theta} | i\in\bI, Y\in\cX, (X,w,Y)\in\Hom_{\cW}(X,Y) \}.
\end{align*}
We say that $\cG$ is \emph{finite} if $\varDelta^{X,\re}$ is finite for all $X\in\cX$ (equivalently, for some $X\in\cX$).

The set of positive and negative real roots are, respectively:
\begin{align*}
\varDelta^{X,\re}_{+} &\coloneq \varDelta^{X,\re} \cap \bN_0^{\theta},  & \varDelta^{X,\re}_{-} &\coloneq \varDelta^{X,\re} \cap (-\bN_0^{\theta}).  
\end{align*}
Given $X\in\cX$, $i\ne j\in\bI$, we set $m^X_{ij} \coloneqq |\varDelta^{X,\re} \cap (\bN_0\alpha_i+\bN_0\alpha_j)|\in\bN\cup\{\infty\}$. 
We say that a semi-Cartan graph $\cG$ is a \emph{Cartan graph} if in addition the following hold:
\begin{itemize}[leftmargin=*]
\item for all $X\in\cX$, $\varDelta^{X,\re} = \varDelta^{X,\re}_{+} \cup \varDelta^{X,\re}_{-}$;
\item for all $X\in\cX$, $i\ne j\in\bI$ such that $m^X_{ij}<\infty$, we have that $(r_ir_j)^{m_{ij}^X}(X)=(X)$.
\end{itemize}

A \emph{root system} over $\cG$ is a family $\cR=(\varDelta^X)_{X\in\cX}$ of subsets $\varDelta^X\subset \bZ^{\theta}$ such that
\begin{align*}
0 & \notin \varDelta^X, & \alpha_i & \in\varDelta^X, & \varDelta^X & \subset \bN_0^{\theta}\cup(-\bN_0^{\theta}), &
s_i^X(\varDelta^X)&=\varDelta^{r_i(X)},
\end{align*}
for all $i\in\bI$ and all $X\in\cX$. Positive and negative roots are defined as usual, and $\cR$ is \emph{finite} if every $\varDelta^X$ is so.
Also, $\cR$ is \emph{reduced} if $\bZ \alpha \cap \varDelta^X=\{\pm\alpha\}$ for all $\alpha\in\varDelta^X$, $X\in\cX$.

According to the definition, a Cartan graph $\cG$ might support different root systems over it. But this is not the case when $\cG$ is finite. 

\begin{theorem}{\cite{HS}*{10.4.7}}\label{thm:finite-root-system-is-real}
If $\cG$ is a finite Cartan graph, then $\cR=(\varDelta^{X,\re})_{X\in\cX}$ is the only reduced root system over $\cG$.
\end{theorem}

Finite root systems are in correspondence with crystallographic arrangements: a subset of hyperplanes in a finite-dimensional $\bR$-vector space satisfying certain properties. We refer to \cites{CH-classif,Cuntz-Lentner} for the precise definition and the correspondence. 

\medbreak

The next result about root systems will be useful throughout the article. 

\begin{theorem}{\cite{CH-rank3}*{Theorem 2.4}}\label{thm:cuntz-heck}
Let  $\gamma_1, \dots \gamma_k \in\varDelta_+^{X}$ be linearly independent roots. Then there exist $Y\in\cX$,  $w\in\Hom (X,Y)$, and $\sigma\in\mathbb{S}_{\bI}$ such that the support of $w(\gamma_i)\in\varDelta_+^{Y}$ is contained in $\{\sigma(1),\dots,\sigma(i)\}$ for each $1\leq i\leq k$. \qed
\end{theorem}

In other words, this result allows us to reduce computations on a set of $k$ linearly independent roots to computations on a root system of rank $k$, obtained as a subsystem of another object in the Weyl class. Moreover, notice that
\begin{align*}
w(\gamma_1)&=\alpha_{\sigma(1)} & \text{ since }\  w(\gamma_1)&\in \bZ\alpha_{\sigma(1)} \cap \varDelta_+^{Y} =\{\alpha_{\sigma(1)} \}.
\end{align*}

\medbreak

There are some standard constructions that produce new finite root systems from old ones, as described in \cite{Cuntz-Lentner}. More relevant for us is the restriction construction; at the level of arrangements, this process fixes a root and then projects all the hyperplanes in the arrangement onto its orthogonal component. We are interested in the description of the associated root system. First, we fix some notation:
\begin{itemize}[leftmargin=*]
\item Let $i\in\bI_{\theta}$. We denote by $\pi_{i}\colon\bZ^{\theta}\to\bZ^{\theta-1}$ the projection given by 
\begin{align*}
\pi_{i}(a_1,\dots,a_{\theta})&=(a_1,\dots,a_{i-1},a_{i+1},\dots,a_{\theta}), & & a_i\in\bZ.
\end{align*}
\item For each $\beta=(b_1,\cdots,b_{\theta-1})\in \bZ^{\theta-1}$, let $\beta'=\tfrac{1}{\gcd(b_i|i\in \bI_{\theta-1})}\beta$.
\end{itemize}

\begin{lemma}{\cite{Cuntz-Lentner}*{3.3}}\label{lem:Cuntz-Lentner}
Let $\cR=(\cC,(\varDelta^a)_{a\in A})$ be a finite root system of rank $\theta$. Fix $a\in A$ and $i\in\bI_{\theta}$. 
Then $\overline{\varDelta}:=\{\pi_i(\alpha)' | \alpha\in\varDelta^a \}$ is the set of roots corresponding to the restriction 
of the hyperplane arrangement of $\cR$ to $\alpha_i^{\perp}$.\qed
\end{lemma}

The classification of finite root systems was achieved in \cite{CH-classif}. Roughly speaking, there exist families of arbitrary rank corresponding to Lie superalgebras and Lie algebras of types $A,B,C,D$, infinite examples in rank two corresponding to triangulations of $n$-agons, and several exceptions in ranks $3\le \theta\le 8$.
According to \cite{Cuntz-Lentner}*{Theorem 3.7}, most of these root systems come from a classical one by restriction.

\subsection{Contragredient Lie superalgebras}\label{subsec:Weyl-gpd} 
Here we recall the construction of contragredient Lie superalgebras over $\Bbbk$ from \cites{AA-super, BGL, Kac-book} and introduce notation.
We fix the following \emph{contragredient data}:

\begin{enumerate}[leftmargin=*,label=\rm{(D\arabic*)}]
\item\label{item:contragredient-1} a matrix $A=(a_{ij})\in\Bbbk^{\bI\times\bI}$;
\item \label{item:contragredient-2} a \emph{parity vector} $\bp=(p_i)\in \bG_2^\bI$\footnote{For our formulas it is more suitable to work with $\bG_2=\{\pm 1\}$ than $\bZ/2$.};
\item \label{item:contragredient-3} a $\Bbbk$-vector space $\fh$ of dimension $2\theta-\rk A$;
\item \label{item:contragredient-4} linearly independent subsets $(\xi_i)_{i\in \bI} \subset \fh^*$ and $(h_i)_{i\in \bI} \subset \fh$ \emph{realizing} the matrix $A$, that is,  $\xi_j(h_i)=a_{ij}$ for all $i,j\in\bI$.
\end{enumerate}
The set $(h_i)_{i\in \bI}$ in \ref{item:contragredient-4} is completed to a basis $(h_i)_{1\leq i\leq2\theta-\rk A}$ of $\fh$.
The Lie superalgebra $\tfg:=\tfg(A,\bp)$ is presented by generators $e_i$, $f_i$, $i\in \bI$, and $\fh$,
with parity given by
\begin{align*}
|e_i|&=|f_i|= \vert i \vert,\quad  i\in\bI, & |h|&=0, \mbox{ for all }h\in\fh,
\end{align*}
subject to the relations, for all $i, j \in \bI$, $h, h'\in \fh$:
\begin{align}
\label{eq:relaciones gtilde}
[h,h']&=0, &[h,e_i] &= \xi_i(h)e_i, & [h,f_i] &= -\xi_i(h)f_i, & [e_i,f_j]&=\delta_{ij}h_i.
\end{align}

There is a unique $\bZ$-grading $\tfg = \mathop{\oplus}\limits_{k\in \bZ}\tfg_k$ such that $e_i\in \tfg_1$, $f_i \in \tfg_{-1}$, $\fh = \tfg_0$. As usual, we denote $\tfn_+= \mathop{\oplus}\limits_{k>0}\tfg_k$ and $\tfn_-= \mathop{\oplus}\limits_{k<0}\tfg_k$.
The family of $\bZ$-homogeneous ideals trivially intersecting  $\fh$ admits a unique maximal ideal $\frr$; clearly $\frr$ splits as a sum of its positive and negative parts: $\frr=\frr_+\oplus \frr_-$. 

The \emph{contragredient Lie superalgebra} associated to the pair $(A,\bp)$ is the Lie superalgebra quotient $\fg(A,\bp):= \tfg(A,\bp) / \frr$. 
For sake of brevity, set $\fg:=\fg(A,\bp)$, still denote by $e_i$, $f_i$, $h_i$ the images in $\fg$ of the generators for $\tfg$, and identify $\fh$ with its image in $\fg$.
By homogeneity, the grading of $\tfg$ induces one in $\fg = \mathop{\oplus}\limits_{k\in \bZ}\fg_k$; thus we have 
$\fg=\fn_+\oplus\fh\oplus\fn_-$ as usual.

\begin{remark} \label{rem:derivedsubalgebra}
The Lie subsuperalgebra $\fg' = [\fg, \fg]$ admits a complement  $\fg=\fg'\oplus \fh_{>\theta}$, where $\fh_{>\theta}$ is the subspace spanned by all $h_j$ with $j>\theta$. In particular, $\fg=\fg'$ if $A$ is non-degenerate. 

As in \cite{Kan}, it will be easier to work with $\fg'$ rather than $\fg$, because the former is generated by the Chevalley generators. 
We denote by $\fh_{\leq \theta}$  the subspace of $\fh$ spanned by all $h_j$ with $j\leq\theta$, thus we have $\fg'=\fn_+\oplus \fh_{\leq \theta} \oplus\fn_-$.
\end{remark}

\begin{remark}
Given $i\in\bI$, denote by $\fg_i$ the Lie subsuperalgebra of $\fg$ generated by $e_i$, $f_i$ and $h_i$. By \cite{BGL}, if a matrix $B$ is obtained from $A$ by rescaling some rows by non-zero scalars, then $\fg(A,\bp)\simeq \fg(B,\bp)$.  Following the convention adopted in \cite{HoytSer}, we always assume that $a_{ii} \in \{0, 2\}$. Hence, for $\fg_i$ we have four possibilities:
\begin{itemize}[label=$\diamondsuit$]
\item $a_{ii}=2$, $p_i=1$: as usual, $\fg_i \simeq \fsl(2)$;
\item $a_{ii}=0$, $p_i=1$: now $\fg_i$ is isomorphic to $ \fH_3$, the Heisenberg algebra;
\item $a_{ii}=2$, $p_i=-1$: here $\fg_i \simeq \osp(2)$;
\item $a_{ii}=0$, $p_i=-1$: in this case $\fg_i$ is isomorphic to $ \supersl{1}{1}$.
\end{itemize}

\end{remark}

Now we recall some features of $\fg$ extracted from \cite{AA-super}:
\begin{enumerate}[leftmargin=*,label=\rm{(\roman*)}]
\item There exists an involution $\widetilde\omega$ of $\tfg$ such that
\begin{align}\label{eq:Chevalley involution}
\widetilde\omega(e_i)&=-f_i, & \widetilde\omega(f_i)&=-p_ie_i, & \widetilde\omega(h)&=-h, & &\mbox{ for all } i\in\bI, h\in\fh.
\end{align}
Clearly $\widetilde\omega(\frr)=\frr$, so $\widetilde\omega$ gives rise to $\omega:\fg \to \fg$, the \emph{Chevalley involution}.

\item The center $\fc$ of $\fg$ coincides with that of $\fg'$; we have $\fc=\{ h\in\fh: \xi_i(h)=0 \mbox{ for all }i\in \bI \}$.

\item The unique $x\in\fn_+$ (respectively, $x\in\fn_-$) such that $[x,f_i]=0$ (respectively, $[x,e_i]=0$) for all $i\in\bI$ is $x=0$.

\item The Lie superalgebra $\tfg$ has a  $\bZ^{\bI}$-grading determined by
\begin{align*}
\deg f_i&=-\alpha_i, & \deg h&=0, & \deg e_i&=\alpha_i,& \mbox{ for all } i&\in \bI, \ h\in \fh.
\end{align*}
One can show that $\frr$ is $\bZ^{\bI}$-homogeneous, hence $\fg=  \fh \oplus \mathop{\oplus}\limits_{\alpha\in\bZ^{\bI}, \alpha\neq 0} \fg_{\alpha}$.
\end{enumerate}

\begin{definition}
The set of roots of $(A, \bp)$ is $\nabla^{(A,\bp)}:=\{\alpha\in\bZ^{\bI} - 0:\fg_\alpha\neq0\}$. 
The set of positive, respectively negative, root is $\nabla_{\hspace{2pt} \pm}^{(A,\bp)}=\nabla^{(A,\bp)}\cap(\pm\bN^{\bI}_0)$.
\end{definition}

Since the Chevalley involution satisfies $\omega(\fg_{\alpha})=\fg_{-\alpha}$ for any $\alpha \in \bZ^\bI$, we get 
\begin{align}\label{eq:prop1 roots}
\nabla^{(A,\bp)}&=\nabla_{\hspace{2pt} +}^{(A,\bp)}\cup\nabla_{\hspace{1pt} -}^{(A,\bp)}, & \nabla_{\hspace{2pt} -}^{(A,\bp)}&= -\nabla_{\hspace{2pt} +}^{(A,\bp)}.
\end{align}

\begin{remark}\label{rem:subset-subalgebra-contragredient}
Given contragredient data \ref{item:contragredient-1}-\ref{item:contragredient-4} and $\bJ \subset \bI$, consider
\begin{itemize}
\item $A_\bJ:=(a_{ij})_{i,j\in\bJ}$, $\bp_\bJ:=(p_i)_{i\in\bJ}$;
\item $\fg_\bJ$ the subalgebra of $\fg=\fg(A,\bp)$ generated by $\fh$, $e_i$, $f_i$, for $i\in\bJ$;
\item $\fh_\bJ$ the subspace of $\fh$ spanned by $h_i$, for $i\in\bJ$;
\item $\fh'_\bJ$ the maximal subspace of $\cap_{i\in\bJ}\ker\xi_i$ trivially intersecting $\fh_\bJ$.
\end{itemize}

By \cite{HoytSer}*{Lemma 2.1}, there is an isomorphism $\fg(A_\bJ,\bp_\bJ)\oplus\fh'_\bJ\cong\fg_\bJ$ which identifies the positive and negative parts of $\fg(A_\bJ,\bp_\bJ)$ with the subalgebras of $\fg(A,\bp)$ generated by $e_i$, respectively $f_i$, for $i \in \bJ$. In particular, we get $\nabla_{\hspace{2pt} \pm}^{(A_\bJ,\bp_\bJ)}=\nabla_{\hspace{2pt} \pm}^{(A,\bp)}\cap\bZ^\bJ$.
\end{remark}

\subsection{Root systems for contragredient Lie superalgebras}\label{subsec:groupoid}
Next we extract from \cite{AA-super} the construction of the Weyl groupoid associated to a contragredient Lie superalgebra.  Let $(A,\bp)$ as in \ref{item:contragredient-1}, \ref{item:contragredient-2}. Following \cites{Kac-book,BGL}, we assume from now on that $A$ satisfies
\begin{align}\label{eq:symmetrizable}
&a_{ij}=0\mbox{ if and only if }a_{ji}=0, & \mbox{for all }j&\neq i\in \bI.
\end{align}

From $(A,\bp)$, we build:
\begin{enumerate}[leftmargin=*,label=\rm{(\arabic*)}]
\item A matrix $C^{(A,\bp)}=\left(c_{ij}^{(A,\bp)}\right)_{i,j\in\bI}$  by $c_{ii}^{(A,\bp)}:= 2$ for each $i\in \bI$ and
\begin{align}\label{eq:definition C}
c_{ij}^{(A,\bp)}&:=-\min\{m\in\bN_0:(\ad f_i)^{m+1} f_j= 0 \},& i&\neq j \in \bI.
\end{align}
Since $\ad f_i$ is locally nilpotent, $C^{(A,\bp)}$ is a well-defined generalized Cartan matrix.
\item For each $i\in \bI$, an involution $s_i^{(A,\bp)}\in \GL(\bZ^{\bI})$ given by
\begin{align}\label{eq:definition si}
s_i^{(A,\bp)}(\alpha_j)&:= \alpha_j-c_{ij}^{(A,\bp)}\alpha_i, & j\in\bI.
\end{align}
\item For each $i\in \bI$, the  $i$-th \emph{reflection} $ r_i(A, \bp) := ( r_i A,  r_i\bp)=\left((\widehat{a}_{jk})_{j,k\in \bI}, (\widehat{p}_j)_{j\in\bI} \right)$, where 
\begin{itemize}[leftmargin=*]
\item$\widehat{p}_j = p_jp_i^{c_{ij}^{(A,\bp)}}$, for all $j \in \bI$;

\item for $j=i$, let $\widehat{a}_{jk}= c_{ik}^{(A,\bp)}a_{ii}-a_{ik}$, for all $k\in\bI$;

\item for $j\ne i$ with $a_{ij}=0$, let $\widehat{a}_{jk}=a_{jk}$, for all $k\in\bI$;

\item for $j\ne i$ with $a_{ij}\ne0$, put $\widehat{a}_{ji}=a_{ji}(c_{ij}^{(A,\bp)}a_{ii}-a_{ij})$, and for all $k\neq i,j$,
\begin{equation}\label{eq:ajk-barra cuando aij neq0}
\widehat{a}_{jk}=c_{ij}^{(A,\bp)}c_{ik}^{(A,\bp)}a_{ji}a_{ii}-c_{ij}^{(A,\bp)}a_{ji}a_{ik}-c_{ik}^{(A,\bp)}a_{ji}a_{ij}+a_{ij}a_{jk}.
\end{equation}

\end{itemize}
\end{enumerate}

\begin{remark}\label{rem:cij-explicit-formula}
For each $a\in\bF_p$, $\widetilde{a}\in\bZ$ denotes the unique integer $0\geq \widetilde{a}\geq1-p$ whose class in $\bF_p$ is $a$.
By \cite{AA-super} the integers $c_{ij}^{(A,\bp)}$ can be explicitly computed as follows.
\begin{enumerate}[leftmargin=*,label=\rm{(\alph*)}]
\item If $a_{ii}=2$ and $a_{ij}\in\bF_p$, then
$$ c_{ij}^{(A,\bp)}=\left\{ \begin{array}{ll} \widetilde{a}_{ij}, &  a_{ij}\in\bF_p, p_i=1\mbox{ or }p_i=-1, \widetilde{a}_{ij}\mbox{ even};
\\ \widetilde{a}_{ij}-p, & a_{ij}\in\bF_p, p_i=-1, \widetilde{a}_{ij}\mbox{ odd};
\\ 1-\frac{3-p_i}{2}p, & a_{ij}\notin\bF_p.
\end{array}\right.  $$

\item If $a_{ii}=0$, then
$$ c_{ij}^{(A,\bp)}=\left\{ \begin{array}{ll} 0, & a_{ij}=0; \\ 1-p, & a_{ij}\neq 0, \, p_i=1; \\ -1,
& a_{ij}\neq 0, \, p_i=-1. \end{array}\right.  $$
\end{enumerate}
\end{remark}

Another crucial result of \cite{AA-super} states that, in this context, we have analogues for Lusztig's isomorphisms on quantum groups.

\begin{theorem}\label{thm:isomorfismo Ti} Let $A\in \Bbbk^{\bI\times \bI}$  satisfying \eqref{eq:symmetrizable},
$\bp\in (\bG_2)^{\bI}$, and $i\in \bI$.
There is a Lie superalgebra isomorphism  $T_i^{(A,\bp)}\colon \fg( r_i A, r_i\bp)\to\fg(A,\bp)$ such that
\begin{align}
T_i^{(A,\bp)}\left(\fg( r_i A, r_i\bp)_\beta\right) &= \fg(A,\bp)_{s_i^{(A,\bp)}(\beta)}, & \mbox{for all }&\beta\in\pm\bN_0^\bI;
\label{eq:Ti mueve grados via si}
\\
T_i^{(A,\bp)}\circ\omega&=\omega\circ T_i^{(A,\bp)}.
\label{eq:Ti conmunta con w}
\end{align}
\end{theorem}

To obtain a reduced root system, we need to disregard roots that are natural multiples of other roots. Consider
\begin{equation}\label{eq:root system (A,p)}
\varDelta_+^{(A,\bp)} \coloneqq \nabla_{\hspace{2pt} +}^{(A,\bp)} - \{k\, \alpha: \, \alpha\in\nabla_{\hspace{2pt} +}^{(A,\bp)}, k\in\bN, k\geq 2\}.
\end{equation}

Now we are ready to state another key result from \cite{AA-super}.
\begin{theorem}\label{thm:root system}
$\cR(\cC_\theta, (\varDelta^{(A,\bp)})_{(A,\bp)\in\cX})$ is a reduced root system. \qed
\end{theorem}

To fully describe $\nabla^{(A,\bp)}$ we need to take into account all the multiples of roots in $\varDelta^{(A,\bp)}$. 

\begin{definition}\label{def:non degenerate odd}
We say that a root $\beta \in \varDelta^{(A,\bp)}$ is odd non-degenerate if there exist $i\in \bI$ and an element $(B,\bq)= r_{i_1} \cdots  r_{i_k}(A,\bp)$ in the Weyl class of $(A,\bp)$ with $b_{ii}=2$, $\bq_i=-1$ such that $s_{i_1}^{(A,\bp)} \cdots s_{i_k}\in \Hom\left((B,\bq), (A,\bp) \right)$ maps $\alpha_i$ to $\beta$. 
We denote by $\varDelta^{(A,\bp)}_{\ond}$ the set of all such roots.
\end{definition}

The existence of odd non-degenerate roots turns out to be the reason behind the existence of integer multiples of roots.

\begin{proposition}{\cite{AA}}\label{prop:non degenerate odd}
Assume that $\dim\fg(A,\bp)<\infty$. Then
\begin{align*}
\nabla^{(A,\bp)}&= \varDelta^{(A,\bp)}\cup\left(2 \varDelta^{(A,\bp)}_{\ond}\right), & &\text{ and } &
\dim\fg(A,\bp)_\beta&=1, \text{ for all }\beta\in\nabla^{(A,\bp)}.\qed
\end{align*}
\end{proposition}

\begin{example}\label{ex:rk2-contr-Lie-superalgebras}
We describe all rank two finite-dimensional contragredient Lie superalgebras:
\begin{enumerate}[leftmargin=*,label=\rm{(\roman*)}]
\item The classical Lie algebras of types $A_2$, $B_2$ and $G_2$, with matrices $\left(\begin{smallmatrix} 2 & -1 \\ -1 & 2 \end{smallmatrix}\right)$, 
$\left(\begin{smallmatrix} 2 & -1 \\ -2 & 2 \end{smallmatrix}\right)$, $\left(\begin{smallmatrix} 2 & -3 \\ -1 & 2 \end{smallmatrix}\right)$ (here, $p>3$). The root systems are the classical ones. 
\item Similarly, $A(0|1)$, $B(0|1)$ are standard root systems, with Cartan matrices $\left(\begin{smallmatrix} 2 & -1 \\ -1 & 2 \end{smallmatrix}\right)$, 
$\left(\begin{smallmatrix} 2 & -1 \\ -2 & 2 \end{smallmatrix}\right)$. The root systems have the same roots as $A_2$ and $B_2$.

\item Let $p=3$, $a\in\Bbbk - \bF_3$. The \emph{Brown algebra} $\brown(2,a)$ is constructed as follows: Set
\begin{align*}
A&=\begin{bmatrix} 2 & -1 \\ a & 2 \end{bmatrix}, &
A'&=\begin{bmatrix} 2 & -1 \\ -1-a & 2 \end{bmatrix},
\end{align*}
so $C^{A}$ and $C^{A'}$ are of type $B_2$. We can check that $r_2(A)=A'$, $r_1(A)=A$, $r_1(A')=A'$,
\begin{align*}
\varDelta^A_+ &= \varDelta^{A'}_+ = \{1, 12, 12^2,2\},
\end{align*}
so $\dim \fg(A)=10$.
\item Again take $p=3$: we recall now the definition of the Lie superalgebra $\mathfrak{brj}(2;3)$. Set
\begin{align*}
A&=\begin{bmatrix} 0 & 1 \\ 1 & 0 \end{bmatrix}, \bp=(1,-1), & 
A'&=\begin{bmatrix} 0 & 1 \\ -2 & 2 \end{bmatrix}, & 
A''&=\begin{bmatrix} 2 & -1 \\ 1 & 0 \end{bmatrix}, \bp''=(-1,-1).
\end{align*}
In this case, 
\begin{align*}
C^{(A,\bp)}&=\begin{bmatrix} 2 & -2 \\ -1 & 2\end{bmatrix}, &
C^{(A',\bp)}&=\begin{bmatrix} 2 & -2 \\ -2 & 2\end{bmatrix}, &
C^{(A'',\bp'')}&=\begin{bmatrix} 2 & -4 \\ -1 & 2\end{bmatrix},
\end{align*}
$r_1(A,\bp)=(A',\bp)$, $ r_2(A,\bp)=(A'',\bp'')$, $ r_2(A',\bp)=(A',\bp)$, $ r_1(A'',\bp'')=(A'',\bp'')$, and
\begin{align*}
\varDelta^{(A,\bp)}_+ & = \{1, 1^22, 1^32^2, 1^42^3, 12,2\}, &
\varDelta^{(A,\bp)}_{\ond} &= \{12, 1^22\},
\\
\varDelta^{(A',\bp)}_+ & = \{1, 1^22, 1^32^2, 12, 12^2, 2\}, &
\varDelta^{(A',\bp)}_{\ond} &= \{1, 12\},
\\
\varDelta^{(A'',\bp'')}_+ & = \{1, 1^42, 1^32, 1^22, 12, 2\}, &
\varDelta^{(A'',\bp'')}_{\ond} &= \{1, 1^22\}.
\end{align*}
Thus $\sd\fg(A,\bp)=10|8$.

\item Finally take $p=5$. The Lie superalgebra $\mathfrak{brj}(2;5)$ admits two possible realizations as contragredient Lie superalgebra. Namely,
\begin{align*}
A&=\begin{bmatrix} 2 & -3 \\ 1 & 0 \end{bmatrix}, \bp=(1,-1), & 
A'&=\begin{bmatrix} 2 & -4 \\ 1 & 0 \end{bmatrix}, \bp'=(-1,-1).
\end{align*}
In this case, 
\begin{align*}
C^{(A,\bp)}&=\begin{bmatrix} 2 & -3 \\ -1 & 2\end{bmatrix}, &
C^{(A',\bp')}&=\begin{bmatrix} 2 & -4 \\ -1 & 2\end{bmatrix},
\end{align*}
$r_2(A,\bp)=(A',\bp')$, $ r_1(A,\bp)=(A,\bp)$, $ r_1(A',\bp)=(A',\bp)$, and
\begin{align*}
\varDelta^{(A,\bp)}_+ & = \{ 1, 1^32, 1^22, 1^52^3, 1^32^2, 1^42^3, 12, 2 \}, &
\varDelta^{(A,\bp)}_{\ond} &= \{1^22, 12\},
\\
\varDelta^{(A',\bp)}_+ & = \{1, 1^42, 1^32, 1^52^2, 1^22, 1^32^2, 12, 2\}, &
\varDelta^{(A',\bp)}_{\ond} &= \{1,1^22\}.
\end{align*}
Thus $\sd\fg(A,\bp)=10|12$.
\end{enumerate}
\end{example}

\begin{remark}\label{rmk:root vectors}
By \Cref{prop:non degenerate odd}, we can fix bases $(e_{\beta})_{\beta\in\nabla^{(A,\bp)}}$ for  $\fn_+$ and  $(f_{\beta})_{\beta\in\nabla^{(A,\bp)}}$  for $\fn_-$ such that
\begin{itemize}[leftmargin=*]
\item $e_{\alpha_i}=e_i$ and  $f_{\alpha_i}=f_i$, for all $i\in \mathbb I$;

\item $e_{\beta}\in \fg_{\beta}$ and $f_{\beta}\in \fg_{-\beta}$, for all $\beta\in\varDelta^{(A,\bp)}_+$; if $\beta = k \alpha_i+\alpha_j$ for some $i\ne j$ and $k\geq 0$, we take $e_\beta=(\ad e_i)^k e_j$ and $f_\beta=(\ad f_i)^k f_j$. 
\item $e_{2\beta}=[e_{\beta},e_{\beta}]$ and  $f_{2\beta}=[f_{\beta},f_{\beta}]$, for all $\beta\in\varDelta^{(A,\bp)}_{+,\ond}$.
\end{itemize}

Let $\beta=\sum_{i\in\bI}a_i\alpha_i\in\nabla^{(A,\bp)}_+$, $\xi_{\beta}\coloneqq \sum_{i\in\bI}a_i\xi_i\in\fh^{\ast}$. Then
\begin{align}\label{eq:action-h-on-root-vector}
[h, e_{\beta}] &=\xi_{\beta}(h)e_{\beta}, & [h, f_{\beta}] &=-\xi_{\beta}(h)f_{\beta}, &\mbox{for all } h&\in\fh. 
\end{align}
\end{remark}

We end this section with a generalization of a well-known result for Lie algebras. The proof follows as an application of \Cref{thm:cuntz-heck}.

\begin{lemma}\label{lem:sum-of-roots}
Assume that $\varDelta_+^{(A,\bp)}$ is finite and $\alpha, \beta \in \varDelta_+^{(A,\bp)}$ are such that $\alpha \notin \bZ \beta$.
\begin{enumerate}[leftmargin=*,label=\rm{(\alph*)}]
\item\label{item:sum-of-roots-i} If $n\in\mathbb N$ is such that $\beta+n\alpha\in \varDelta_+^{(A,\bp)}$, then $n<2p$ and there exists $c\in\Bbbk^{\times}$ such that 
$$(\ad e_{\alpha})^n e_{\beta}= c e_{\beta+n\alpha} .$$
\item\label{item:sum-of-roots-ii} If $n\in\mathbb N$ is such that $\beta-n\alpha\in \varDelta_+^{(A,\bp)}$, then $n<2p$ and there exists $c\in\Bbbk^{\times}$ such that 
$$(\ad e_{\alpha})^n f_{\beta}= c f_{\beta-n\alpha} .$$

\item\label{item:sum-of-roots-iii} Assume moreover that $\alpha \notin \varDelta^{(A,\bp)}_{\ond}$. If $n$ is as in any of the items above, then $n<p$.
\end{enumerate}
\end{lemma}

\begin{proof}
\ref{item:sum-of-roots-i}, \ref{item:sum-of-roots-ii} Since $\alpha$ and $\beta$ are linearly independent, by Theorem \ref{thm:cuntz-heck} there exists a pair $(B,\bp')$ with two simple roots $\beta_1, \beta_2$, and $w\in\Hom ( (A, \bp), (B,\bp'))$ 
such that 
\begin{align*}
\omega (\alpha) = \beta_1, && \omega (\beta)\in \varDelta_+^{(B,\bp')} \cap \{\bN_0 \beta_1 + \bN_0 \beta_2\}.
\end{align*}
So it is enough to verify the claims for $(A,\bp)$ of rank $\theta=2$, which is easily achieved from a case-by-case analysis, see \Cref{ex:rk2-contr-Lie-superalgebras}.

For \ref{item:sum-of-roots-iii} we can reduce to $\theta=2$ again and use that $\beta_1 \notin \varDelta^{(B,\bp')}_{\ond}$.
\end{proof}

\section{Lie algebras in symmetric tensor categories}\label{sec:Lie-algs-STC}

We recall some notions and basic results related to symmetric tensor categories, Lie algebras in this broad context, and the Verlinde category.

\subsection{Symmetric tensor categories}
A \emph{symmetric tensor category} is an abelian $\bk$-linear category which admits a rigid symmetric monoidal structure such that the tensor product is bilinear on hom-spaces and the unit object has a one-dimensional endomorphism space.

A \emph{pre-Tannakian} category is a symmetric tensor category where all objects have finite length, which then implies that all Hom spaces are finite-dimensional. Notice that we adopt the terminology of \cite{CEO}.

\smallbreak

The ind-completion $\indcat{\cC}$ of a symmetric tensor category $\cC$ is defined as the closure of $\cC$ under filtered colimits. 
This completion $\indcat{\cC}$ is a $\bk$-linear abelian category with an exact and symmetric tensor product, and there is universal exact symmetric embedding $\cC \hookrightarrow \indcat{\cC}$, see \cite{Kash-Shap}. We shall refer objects in $\indcat{\cC}$ as ind-objects of $\cC$.
When $\cC$ is moreover fusion, that is, finite and semisimple, the ind-objects are (possibly infinite) direct sums of simple objects in $\cC$. The genuine objects of $\cC$ are then recovered as ind-objects of finite length. If $X$ and $Y$ are ind-objects, we imprecisely write $\Hom_{\cC} (X,Y)$ instead of $\Hom_{\indcat{\cC}} (X,Y)$.

\subsection{The category \texorpdfstring{$\Rep(\balp_p)$}{}}\label{subsec:alphap}
Let $\balp_p$ denote the Frobenius kernel of the additive group scheme. This group scheme is represented by the commutative and cocommutative Hopf algebra $\bk[t]/(t^p)$ where $t$ is primitive. By self-duality, we can safely identify $\Rep(\balp_p)$ with the symmetric tensor category of finite dimensional representations of $\bk[t]/(t^p)$ over $\bk$. 

Isomorphism classes of indecomposables in $\Rep(\balp_p)$ are represented by nilpotent Jordan blocks $L_i$ of size $i$ with $1\leq i \leq p$. Explicitly, $L_i$ has an ordered basis $\{v_1, \dots, v_i\}$ such that $v_{j+1}=t v_{j}$ for $1\leq j <i$ and $tv_i=0$. Such a basis is called \emph{cyclic}; in this situation we write $L_i=\bk\{v_1, \dots, v_i\}=\langle v_1\rangle$ and say that $v_1$ is a cyclic generator. Notice that an element in $\Hom_{\alpha_p}(\langle v_1\rangle, V)$ is determined by its value on $v_1$. 

The indecomposable objects are self-dual, and an isotypic decomposition of $L_i\ot L_j$ is known from \cite{Green}. In particular $L_1$ is the monoidal unit and $L_p \otimes L_p\simeq pL_p$. For later use, we record an explicit decomposition of some tensor products.

\begin{example}\label{ex:alphap-isotypic-L2} 
Assume $p>2$. Consider two copies of $L_2$, with cyclic bases $\{v_1, v_2\}$ and $\{w_1, w_2\}$ respectively. Then $L_2\otimes L_2=L_1 \oplus L_3$ where \begin{align*}
L_1&= \langle v_1\otimes w_2 - v_2 \otimes w_1\rangle, & L_3&=\langle v_1\otimes w_1\rangle .
\end{align*}
\end{example}

\begin{example}\label{ex:alphap-isotypic-L3-p>3} 
Consider two copies of $L_3$, with cyclic bases $\{v_1, v_2, v_3\}$ and $\{w_1, w_2,w_3\}$. 

\begin{itemize}[leftmargin=*]
\item If $p>3$, then $L_3\otimes L_3=L_1 \oplus L_3\oplus L_5$ where
\begin{align*}
L_1& = \langle v_1\otimes w_3 - v_2 \otimes w_2+v_3\otimes w_1 \rangle, & 
L_3& = \langle v_1\otimes w_2 -v_2\otimes w_1\rangle, &
L_5& = \langle  v_1\otimes w_1\rangle.
\end{align*}

\item If $p=3$, then $L_3\otimes L_3=L_3^{(1)} \oplus L_3^{(2)}\oplus L_3^{(3)}$ where 
\begin{align*}
L_3^{(1)}& = \langle  v_2\otimes w_2 \rangle, & 
L_3^{(2)}& = \langle v_1\otimes w_2 -v_2\otimes w_1 \rangle, &
L_3^{(3)}& = \langle  v_1\otimes w_1 \rangle.
\end{align*}
\end{itemize}
\end{example}

\begin{example}\label{ex:alphap-isotypic-L4-p>5}
Consider two copies of $L_4$, with bases $\{v_1, v_2, v_3,v_4\}$ and $\{w_1, w_2,w_3,w_4\}$ respectively.
\begin{itemize}[leftmargin=*]
\item If $p>5$, then $L_4\otimes L_4=L_1 \oplus  L_3 \oplus  L_5 \oplus  L_7$, where 
\begin{small}
\begin{align*}
L_1& = \langle v_1\otimes w_4 - v_2 \otimes w_3+v_3\otimes w_2-v_4\otimes v_1 \rangle, &
L_5& = \langle  v_1\otimes w_2- v_2\otimes w_1 \rangle,
\\
L_3& = \langle 3v_1\otimes w_3 -4v_2\otimes w_2+3v_3\otimes w_1 \rangle, &
L_7& = \langle v_1\otimes w_1 \rangle .
\end{align*}
\end{small}

\item If $p=5$, then $L_4\otimes L_4=L_1\oplus 3 L_5$, where 
\begin{align*}
L_1& = \langle v_1\otimes w_4 - v_2 \otimes w_3+v_3\otimes w_2-v_4\otimes v_1\rangle .
\end{align*}
\end{itemize}
\end{example}

For reasons that will become evident later, we are particularly interested in tensor products of the form $L_3\otimes L_j$ for $j$ at most $4$. Next, we describe explicit direct sum decompositions.

\begin{example}\label{ex:alphap-isotypic-L3-L2-p>3}
Consider indecomposables $L_3=\bk\{v_1, v_2, v_3\}$ and $L_2=\bk\{w_1, w_2\}$. 
\begin{itemize}[leftmargin=*]
\item Assume $p>3$. Then $L_3\otimes L_2=L_2 \oplus  L_4$, where 
\begin{align*}
L_2&=\langle 2v_1\otimes w_2 - v_2 \otimes w_1\rangle, & L_4&=\langle v_1\otimes w_1\rangle.
\end{align*}

\item Assume $p=3$. Then $L_3\otimes L_2=L_3^{(1)} \oplus L_3^{(2)}$, where
\begin{align*}
L_3^{(1)}& = \langle v_1\otimes w_1 \rangle, &
L_3^{(2)}& = \langle v_2\otimes w_1 \rangle.
\end{align*}
\end{itemize}
\end{example}

\begin{example}\label{ex:alphap-isotypic-L3-L4-p>5}
Consider indecomposables $L_3=\bk\{v_1, v_2, v_3\}$ and $L_4=\bk\{w_1, w_2, w_3, w_4\}$.

\begin{itemize}[leftmargin=*]
\item Assume $p>5$. Then $L_3\otimes L_4=L_2 \oplus  L_4\oplus L_6 $, where 
\begin{align*}
L_2&=\langle v_1\otimes w_3-2 v_2\otimes w_2 + 3  v_3 \otimes w_1\rangle, & L_4&=\langle 2v_1\otimes w_2 - 3v_2\otimes w_1\rangle, & L_6&=\langle v_1\otimes w_1\rangle.
\end{align*}

\item Assume $p=5$. Then $L_3\otimes L_4=L_2 \oplus 2 L_5$, where
\begin{align*}
L_2&=\langle v_1\otimes w_3-2 v_2\otimes w_2 + 3  v_3 \otimes w_1\rangle.
\end{align*}
\end{itemize}
\end{example}

\subsection{Semisimplification and the Verlinde category}\label{subsec:verp}

If $\cC$ is a symmetric tensor category, a tensor ideal $\cI$ is a collection of subspaces 
\begin{align*}
\cI&=\{\cI(X,Y)\subseteq \Hom_\cC(X,Y) | X, Y \text{ objects in }\cC\},
\end{align*}
which is compatible with compositions and tensor products. The \emph{quotient category} $\overline{\cC}$ of $\cC$ by $\cI$ is the category whose objects are the same as $\cC$ and the spaces of morphisms are $\Hom_{\overline{\cC}}(X,Y)\coloneqq\Hom_\cC(X,Y)/ \cI(X,Y)$. Notice that this construction is compatible with the $\bk$-linear monoidal structure and the tensor product of the original category.

As in \cite{EO-ss}, see also the references therein, we may take a spherical category $\cC$ and the tensor ideal of \emph{negligible} objects. The quotient $\overline{\cC}$ obtained in this case is known as the \emph{semisimplification} of $\cC$ since it turns out to be semisimple: The simple objects of $\overline{\cC}$ are the indecomposable objects in $\cC$ with non-zero categorical dimension.

\medbreak 

The Verlinde category $\Verp$ is defined as the semisimplification of the spherical tensor category $\Rep(\balp_p)$, see \cites{GelKaz, Ost2}. This is a symmetric fusion category equipped with a symmetric monoidal functor $\Rep(\balp_p)\to \Verp$ which fails to be right or left exact. For instance, since the categorical dimension of $L_p$ is $p=0$, it is mapped to zero in $\Verp$. Up to isomorphism, the simple objects of $\Verp$ are the images $\ttL_i$ of the indecomposables $L_i$ with $1\leq i \leq p-1$ of $\Rep (\balp_p)$, and the fusion rules in $\Verp$ are 
\begin{align}\label{eq:verp-fusionrules}
\ttL_i \ot \ttL_j = \bigoplus_{k=1}^{\min\{i, j, p-i, p-j\}} \ttL_{\vert i-j \vert + 2k-1}.
\end{align}
Also, the fusion subcategory generated by $\ttL_1$ and $\ttL_{p-1}$ is equivalent, as a symmetric tensor category, to $\sVec$, see \cites{Ost2,Kan}.

\subsection{Lie algebras in symmetric tensor categories}\label{subsec:operadic} 
We recall now different flavors of \emph{Lie algebras} in a (strict) symmetric tensor category $\cC$ following \cite{Etingof}. First, an \emph{operadic Lie algebra} is an object $\fg$ in $\cC$ equipped with a morphism $b\colon \fg\otimes\fg \to \fg$ that satisfies the anti-symmetric and Jacobi identities: 
\begin{align}\label{eq:operadic}
b\circ(\id_{\fg\ot \fg} + c_{\fg, \fg}) = 0,  \qquad b\circ (b\ot \id_{\fg})\circ(\id_{\fg^{\ot 3}} + (123)_{\fg^{\ot 3}} +(132)_{\fg^{\ot 3}}) = 0,
\end{align}
where the maps $(123)_{\fg^{\ot 3}}, (132)_{\fg^{\ot 3}}\colon\fg^{\ot 3}\to \fg^{\ot 3}$ are obtained the action of the symmetric group. We denote by $\OLie(\cC)$ the category of operadic Lie algebras in $\cC$. Operadic Lie algebras are just called Lie algebras in \cite{Rumynin} and other previous papers. However, especially when working in positive characteristic, one needs to pay special attention to some details, as explained next.

\medbreak

For a Lie algebra $\fg$, consider two quotients of the tensor algebra $\op{T}(\fg)$:
\begin{itemize}[leftmargin=*]\renewcommand{\labelitemi}{$\diamond$}
\item The \emph{universal enveloping algebra} $\op{U}(\fg)$ is the quotient by the two-sided ideal generated by the image of the map
\begin{align}\label{eq:defn-enveloping-algebra}
(-b,\id_{\fg\ot\fg}- c_{\fg,\fg})\colon \fg\ot\fg\to \fg\oplus\fg\ot\fg\subseteq \op{T}(\fg).
\end{align}
\item The symmetric algebra $\op{S}(\fg)$ is the quotient of $\op{T}(\fg)$ by the image of $\id_{\fg\ot\fg}- c_{\fg,\fg}$.
\end{itemize}
As in the category of vector spaces, there is a natural map $\eta \colon \op{S}(\fg) \to \gr \op{U}(\fg)$ of ind-algebras in $\cC$, which can fail to be an isomorphism. 
For example, it is known that for $\cC=\Vec$ and $p=2$, for $\eta$ to be an isomorphism one needs to impose the additional condition $b(x\ot x)=0$ for all $x$. Also, for $\cC=\sVec$ and $p=3$, we have to assume $b(b(x\ot x)\ot x)=0$ for all odd $x$. 

Motivated by this fact, an operadic Lie algebra $\fg$ is said \emph{PBW} if the canonical map $\eta \colon \op{S}(\fg) \to \gr \op{U}(\fg)$ is an isomorphism.

\medbreak

It is not a trivial task to decide whether a given operadic Lie algebra in a symmetric tensor category is PBW.
For $\cC=\Ver_p$, Etingof introduced in\cite{Etingof} the \emph{$p$-Jacobi identity} for any $p\geq5$, a degree-$p$ relation that generalizes the conditions $b(x\ot x)=0$ and $b(b(x\ot x)\ot x)=0$ needed for characteristics $2$ and $3$, respectively.
By \cite{Etingof}*{Theorem 6.6, Corollary 6.7}, an operadic Lie algebra in $\Ver_p$ is PBW if and only if it satisfies the $p$-Jacobi identity. As a first application, we have the following result:

\begin{lemma}\label{contragredient-pJacobi}
Contragredient operadic Lie superalgebras satisfy the PBW theorem in any characteristic. 
In other words, for any contragredient Lie superalgebra $\fg=\fg(A,\bp)$, the natural map  $\eta \colon \op{S}(\fg) \to \gr \op{U}(\fg)$ is an isomorphism. 
\end{lemma}

\begin{proof}
By \cite{Etingof}*{Corollary 4.12}, it is enough to show that if $p=2$ then $[x,x]=0$ for all $x \in \fg$ and that when $p=3$, we have $[x,[x,x]]=0$ for all odd $x$. 
In any case, we can assume that $x$ is homogeneous of nonzero degree; moreover, via the Chevalley involution, we reduce to the case $x\in \fn_+$. 

When $p=2$, it is enough to show that $[f_i,[x,x]]=0$ for all $i\in \bI$, which is a consequence of the Jacobi identity:
\begin{align*}
[f_i, [x,x]] = [[f_i,x],x]+[x,[f_i,x]] = 2[[f_i,x],x]=0.
\end{align*}

For the case $p=3$, we need to show that  $[f_i,[x, [x,x]]]=0$ for $i \in \bI$ and any odd $x\in \fn_+$. When $f_i$ is odd, using the Jacobi identity and the fact that $[f_i,x]$ is even, we get
\begin{align*}
[f_i,[x, [x,x]]]&=[[f_i,x],[x,x]] - [x,[f_i, [x,x]]] \\
&= [[[f_i,x],x],x] +[x,[[f_i,x],x]] - [x, [[f_i,x],x] - [x, [f_i,x]]]
=0.
\end{align*}
Similarly, when $f_i$ is even, we see that
\begin{align*}
[f_i,[x, [x,x]]]&=[[f_i,x],[x,x]] + [x,[f_i, [x,x]]] \\
&= [[[f_i,x],x],x] -[x,[[f_i,x],x]] + [x, [[f_i,x],x] + [x, [f_i,x]]]
=0,
\end{align*}
as claimed.
\end{proof}

Fix an operadic Lie algebra $\fg$ with bracket $b$ in a strict symmetric tensor category $\cC$.
A form on $\fg$ is a map $B\colon\fg\ot\fg\to \one$ in $\cC$. We say that the form is:
\begin{itemize}[leftmargin=*]
\item \emph{symmetric} if it remains unchanged when composed with the braiding $\fg\ot \fg \to \fg \ot \fg$;
\item \emph{invariant} if $B\circ(b\ot \id_{\fg})\circ(\id_{\fg^{\ot 3}}+(123)_{\fg^{\ot 3}})=0$;
\item \emph{non-degenerate} if its image under the natural adjunction $\Hom_\cC(\fg\ot\fg, \one)\simeq \Hom_\cC(\fg,\fg^*)$ is an isomorphism.
\end{itemize} 

\begin{lemma}\label{lemma:bilinear-forms-ss}
Let $\fg$ be an operadic Lie algebra with a form $B:\fg\ot\fg\to \one$.
Let $\ov{B}:\ov{\fg}\ot\ov{\fg}\to\one$ denote the form on $\ov{\fg}$ induced by $B$ under semisimplification.
If $B$ is symmetric, respectively, invariant, and non-degenerate, then so is $\ov{B}$.
\end{lemma}
\begin{proof}
The property of being symmetric (respectively, invariant) is preserved because the semisimplification functor is additive. Non-degeneracy is also hereditary since the diagram 
\[\begin{tikzcd}
\Hom_\cC(\fg\ot\fg, \one)\arrow{rr}{\sim} \arrow{d} && \Hom_{\cC}(\fg,\fg^*) \arrow{d}\\
\Hom_{\ov{\cC}}(\ov{\fg}\ot\ov{\fg}, \one) \arrow{rr}{\sim}&& \Hom_{\ov{\cC}}(\ov{\fg},\ov{\fg}^*).
\end{tikzcd}\]
commutes.
\end{proof}

\section{Semisimplification of Lie algebras with a derivation}\label{sec:ss-Lie-algebras}
As in \cite{Kan}, if $\fg$ is a finite dimensional Lie algebra over $\bk$ endowed with a derivation $\partial$ of order at most $p$, then $\fg$ becomes a Lie algebra in $\Rep (\balp_p)$ by letting $t$ act via $\partial$. 
Applying the semisimplification $\Rep (\balp_p)\to \Ver_p$, we get an operadic Lie algebra in $\Verp$ since the functor is braided monoidal.

We are interested in understanding the result of this process when the input is a contragredient Lie algebra $\fg=\fg(A)$ with an inner derivation $\partial=\ad x$, for some $x\in \fg$. 
If moreover, $x$ is homogeneous with respect to the grading by $ \varDelta^A$, using \Cref{thm:cuntz-heck}, we may assume that $x=e_i$ for some $i\in\bI$. Let us fix some terminology.

\begin{notation} We denote by 
 $\sfS \colon \Rep(\balp_p) \to \Verp$ the semisimplification functor.
If $\fg$ is a finite dimensional Lie algebra and $x\in \fg$ is such that $\ad x$ has order at most $p$, then
\begin{itemize}[leftmargin=*]
\item $(\fg, x)$ denotes the Lie algebra in $\Rep(\balp_{p})$ obtained from $\fg$ by letting $t$ act via $\ad x$;
\item $\sfS (\fg, x)$ is the operadic Lie algebra in $\Ver_p$ obtained from semisimplification of $(\fg, x)$.
\end{itemize}
\end{notation}

Next we verify that for a finite dimensional contragredient Lie algebra $\fg(A)$, the Chevalley generators $e_i$, $f_i$ yield suitable derivations.

\begin{lemma}\label{lem:order-inner-derivation}
Let $A$ be such that $\dim \fg(A)<\infty$. Then $(\ad e_i)^{p}=(\ad f_i)^{p}=0$ for all $i\in\bI_\theta$.
\end{lemma}
\begin{proof}
It is known that $(\ad x)^{p}$ is a derivation for all $x\in \fg$. Hence, to show that $(\ad e_i)^{p}$ annihilates $\fn_+$, 
it is enough to check that $(\ad e_i)^{p}e_j=0$ for all $j\in\bI$, which follows from  \Cref{rem:cij-explicit-formula}.
Also, for any $x\in\fh=\fg_0$, we have $(\ad e_i)^2 x \in (\ad e_i)\Bbbk e_i=0$.
Thus ${(\ad e_i)^{p}}_{|\fn_+\oplus\fh}=0$. Applying the Chevalley involution we get that ${(\ad f_i)^{p}}_{|\fn_-\oplus\fh}=0$.

To show that $(\ad e_i)^{p}$ annihilates $\fn_-$, consider a nonzero homogeneous $x\in\fg_{-\beta}$,  where $\beta\in\varDelta_+^A$. If $\beta=\alpha_i$, then $x\in \Bbbk f_i$, and since
\begin{align*}
(\ad e_i)^3 f_i &= (\ad e_i)^2 h_i = - a_{ii}[e_i,e_i]=0,
\end{align*}
we also have $(\ad e_i)^3 x=0$.
For the case $\beta\ne \alpha_i$, consider $\beta':=s_i(\beta)\in \varDelta_+^{ r_i A}$; we know that, up to a nonzero scalar, $x=T_i^{A}(\fb_{\beta'})$. Then
\begin{align*}
(\ad e_i)^{p}x &= \left(\ad T_i^{A}(\fb_i)\right)^{p}T_i^{A}(\fb_{\beta'}) = T_i^{A} \left( (\ad \fb_i)^{p}\fb_{\beta'} \right)=0.
\end{align*}
Thus ${(\ad e_i)^{p}}_{|\fn_-}=0$. Applying the Chevalley involution, we also have ${(\ad f_i)^{p}}_{|\fn_+}=0$.
\end{proof}

These two derivations turn out to yield isomorphic Lie algebras in $\Rep(\balp_p)$.

\begin{lemma} \label{lem:(g,ei)=(g,fi)}
Let $A\in \Bbbk^{\theta\times\theta}$ such that $\fg= \fg(A)$ is finite dimensional, and fix $i\in\bI_{\theta}$. Then the Chevalley involution 
$\omega \colon (\fg, e_i )\to (\fg, -f_i)$ is an isomorphism of Lie algebras in $\Rep(\balp_{p})$.
\end{lemma}

\begin{proof}
Since $\omega\colon \fg \to \fg$ is an isomorphism of Lie algebras, it is enough to show that $\omega\colon (\fg, e_i )\to (\fg, -f_i)$ is a morphism in $\Rep (\balp_{p})$, which follows directly from $\omega (e_i)=-f_i$.
\end{proof} 

\begin{remark}\label{rem:reduction-to-connected}
Let $\bJ\subseteq \bI$ be the connected component of the Dynkin diagram containing a fixed index $i\in \bI_{\theta}$, and denote by $\widehat \bJ$ the complement of $\bJ$ in $\bI$. Then, as explained in \Cref{rem:subset-subalgebra-contragredient}, we have  $\fg=\fg(A_{\bJ})\oplus \fg(A_{\widehat \bJ})$, so
\begin{align*}
(\fg,e_i)&=(\fg(A_{\bJ}),e_i)\oplus (\fg(A_{\widehat \bJ}),e_i), & \sfS(\fg,e_i)&=\sfS(\fg(A_{\bJ}),e_i)\oplus \sfS(\fg(A_{\widehat \bJ}),e_i).
\end{align*}
Furthermore, since $\restr{(\ad e_i)}{\fg(A_{\widehat \bJ})}=0$, we see that $\sfS(\fg(A_{\widehat \bJ}),e_i)\simeq \fg(A_{\widehat \bJ})$ is an ordinary Lie algebra in $\Vec \subset \Ver_p$, similarly $(\fg(A_{\widehat \bJ}),e_i)\in \Vec\subset \Rep(\balp_p)$. 
\end{remark}

As a consequence, for most of the arguments in this section, we only need to consider matrices $A$ with connected Dynkin diagram.

\subsection{The structure of the Lie algebra \texorpdfstring{$(\fg' (A), e_i)$}{}}\label{subsec:g-in-repalphap}
In this subsection, we fix a finite-dimensional contragredient Lie algebra $\fg=\fg (A)$ and a Chevalley generator $e_i$, and we describe the structure of the Lie algebra $(\fg', e_i)$ in $\Rep (\balp_p)$, see \Cref{rem:derivedsubalgebra}.

As a first step, we show that the isotypic decomposition of $(\fg', e_i) \in \Rep (\balp_p)$ is determined by the roots. To describe it, we need to fix some terminology.

\begin{notation}\label{notation:strings} Fix $i\in\bI_{\theta}$ and a matrix $A\in \Bbbk^{\theta\times\theta}$. For each $j\in\bI_{p}$ consider 
\begin{align*}
\varDelta_{+,j}^{A} &= \left\{ \beta \ne \alpha_i  \vert \beta + k \alpha_i \in \varDelta_+^{A} \text{ for all } 0\leq k < j   \text{ and }  
\beta-\alpha_i, \beta + j \alpha_i \notin \varDelta_+^{A} \right\}, & a_j &=|\varDelta_{+,j}^{A}|.
\end{align*}
If $\beta \in \varDelta_{+,j}^{A}$, we say that $\{\beta + k \alpha_i \colon 0\leq k <j\}$ is a maximal $\alpha_i$-string. We also say that $\beta$ generates the string, and that $j$ is the length of the string.

The roots that generate $\alpha_i$-strings of length less than $p$ will play an important role. Let
\begin{align}\label{eq:Delta-min}
\varDelta_{+,\min}^{A} &=\bigcup_{j\in\bI_{p-1}} \varDelta_{+,j}^{A}, &  \varDelta_{-,\min}^{A}&=-\varDelta_{+,\min}^{A}, &  \varDelta_{\min}^{A}&=\varDelta_{+,\min}^{A} \bigcup \varDelta_{-,\min}^{A}.
\end{align}
\end{notation}

\begin{remark} \label{rem:strings-partition} Assume $\fg(A)$ is finite dimensional and fix $i\in\bI_\theta$.
\begin{itemize}[leftmargin=*]
\item Each positive root different from $\alpha_i$ belongs to a unique maximal $\alpha_i$-string. In fact, the existence of such a string follows from  \Cref{lem:sum-of-roots} \ref{item:sum-of-roots-iii} (with $\alpha=\alpha_i$), and the uniqueness follows from maximality.
\item Given $j\ne i$, the simple root $\alpha_j$ generates a maximal $\alpha_i$-string of length $1-c_{ij}^A$, see \eqref{eq:definition C}. 
\end{itemize} 
\end{remark}

\begin{proposition}\label{prop:semisimplification-multidimension} 
Fix $i\in\bI_{\theta}$ and a matrix $A\in \Bbbk^{\theta\times\theta}$ such that $\fg= \fg(A)$ is finite dimensional. 
\begin{enumerate}[leftmargin=*,label=\rm{(\alph*)}]
\item\label{item:semisimplification-multidimension-i} For each $\beta \in \varDelta_{+,j}^{A}$, the subspaces 
\begin{align*}
M_{\beta} &:= \oplus_{k=0}^{j-1} \Bbbk e_{\beta+k\alpha_i}, &
N_{\beta} &:= \oplus_{k=0}^{j-1} \Bbbk f_{\beta+k\alpha_i},
\end{align*}
are $\balp_{p}$-submodules of $(\fg', e_i)$, and both are isomorphic to $L_j$.
\item\label{item:semisimplification-multidimension-ii} If $a_{ii}=2$, then $\bk f_i\oplus \bk h_i \oplus \Bbbk e_i$ is a submodule of $(\fg', e_i)$ isomorphic to $L_3$, any one-dimensional subspace of $\ker \xi_i\cap \fh_{\leq \theta}$ is a submodule isomorphic to $L_1$ and 
\begin{align*}
(\fg', e_i) &\simeq L_1^{2a_1 + \theta -1} \oplus L_3^{2a_3+1}\oplus\left(\oplus_{j\in\bI_{p}-\{1,3\}} L_j^{2a_j}\right) \text { in }\Rep(\balp_p).
\end{align*}
\item\label{item:semisimplification-multidimension-iii} If $a_{ii}=0$ and there exists $j$ such that $a_{ji}\ne 0$, then $\bk f_i\oplus \bk h_i $ is a submodule of $(\fg', e_i)$ isomorphic to $L_2$, there is $h \in \fh_{\leq\theta}$ such that  $\bk h \oplus \bk e_i $ is a submodule isomorphic to $L_2$, any one-dimensional subspace of $\ker \xi_i \cap \fh_{\leq \theta }$ different from $\bk h_i$ is a submodule isomorphic to $L_1$ and 
\begin{align*}
(\fg', e_i)&\simeq L_1^{2a_1 + \theta -2} \oplus L_2^{2a_2+2}\oplus\left(\oplus_{j\in\bI_{p}-\{1,2\}} L_j^{2a_j}\right) \text { in }\Rep(\balp_p).
\end{align*}

\end{enumerate}
\end{proposition}

\begin{proof}
\ref{item:semisimplification-multidimension-i} Fix $j \in \bI_p$ and $\beta \in \varDelta_{+,j}^{A}$. Using that $\fg$ is $\mathbb Z^{\theta}$-graded, since $\beta-\alpha_i$ and $\beta + j \alpha_i$ are not roots, we have
\begin{align*}
(\ad e_i)f_{\beta} &=0, & (\ad e_i)e_{\beta+(j-1)\alpha_i}&=0;
\end{align*}
so $N_{\beta}$ and $M_{\beta}$ are $\balp_{p}$-submodules of $\fg$. Using these last computations and Lemma \ref{lem:sum-of-roots}, it becomes evident that these submodules are isomorphic to the nilpotent Jordan block $L_j$.

To prove \ref{item:semisimplification-multidimension-ii} and \ref{item:semisimplification-multidimension-iii}, we first observe that 
$\oplus_{\beta\ne\alpha_i} \Bbbk e_{\beta}$ and $\oplus_{\beta\ne\alpha_i} \Bbbk f_{\beta}$ are $\balp_{p}$-submodules of $\fg$. Indeed, by \Cref{rem:strings-partition} we have
\begin{align*}
\oplus_{\beta\ne\alpha_i} \bk e_{\beta} &= \bigoplus_{j\in\bI_{p}}\left( \oplus_{\beta\in \varDelta_{+,j}^{A}} M_{\beta} \right), &
\oplus_{\beta\ne\alpha_i} \bk f_{\beta} &= \bigoplus_{j\in\bI_{p}}\left( \oplus_{\beta\in \varDelta_{+,j}^{A}} N_{\beta} \right).
\end{align*}
The subspace $\bk f_i\oplus\fh\oplus \bk e_i$ is also an $\balp_{p}$-submodules of $\fg$, so
$$\fg' = \left(\oplus_{\beta\ne\alpha_i} \Bbbk e_{\beta}\right) \oplus \left(\oplus_{\beta\ne\alpha_i} \bk f_{\beta}\right)\oplus \left(\Bbbk f_i \oplus \fh_{\leq\theta} \oplus \bk e_i\right)$$
as an $\balp_{p}$-module. It only remains to compute the isotypic components of $\Bbbk f_i\oplus\fh_{\leq\theta}\oplus \bk e_i$. 

First, we assume that $a_{ii}=2$. In this case $\fh_{\leq\theta} = \bk h_i \oplus (\ker \xi_i\cap \fh_{\leq\theta})$ as vector spaces and $\{f_i,h_i,e_i\}\simeq \mathfrak{sl}_2$, hence $\bk f_i\oplus \bk h_i \oplus \bk e_i \simeq L_3$ as an object in $\Rep(\balp_p)$. 
Also, $(\ker \xi_i)\cap \fh_{\leq \theta}$ is a sum of $\theta-1$ copies of $L_1$, so \ref{item:semisimplification-multidimension-ii} holds.

Now consider $a_{ii}=0$. By hypothesis, there is $h\in\fh_{\leq \theta}$ such that $\xi_i (h)=1$. Thus we have a linear complement $\fh_{\leq \theta}=\bk h \oplus (\ker \xi_i \cap \fh_{\leq \theta})$. Since $[h,e_i]=e_i$, the subspace spanned by $h$ and $e_i$ is an 
$\balp_{p}$-submodule of $\fg$ isomorphic to $L_2$. Also, since $[h_i,e_i]=0$, the subspace spanned by $h_i$ and $f_i$ is an 
$\balp_{p}$-submodule of $\fg$ isomorphic to $L_2$. Finally, any complement of $\Bbbk h_i$ in $\ker \xi_i\cap \fh_{\leq \theta}$ is a sum of $\theta-2$ copies of $L_1$ and \ref{item:semisimplification-multidimension-iii} follows.
\end{proof}

We show next that the submodules $M_\beta$, $N_\beta$ can be Lie-generated by submodules associated to simple roots.

\begin{proposition}\label{prop:generation-degree1-alphap}
If $\beta$ is an $\alpha_i$-string generator, then $M_\beta$ is contained in the $\Rep(\balp_p)$-Lie subalgebra of $(\fg', e_i)$ generated by the submodules $M_{\alpha_j}$, where $j\ne i$. 
\end{proposition}

\begin{proof}
The proof is by induction in $\height (\beta)$. The base case $\height (\beta)=1$, i.e. $\beta$ a simple root, is clear. 

For $\height( \beta) >1$, there exist positive roots $\alpha$ and $\gamma$ such that $\beta = \alpha + \gamma$; notice that both roots have smaller height and are independent with $\alpha_i$. Now,  $\alpha$ and $\gamma$ belong to some $\alpha_i$-string, say $\alpha= \widetilde \alpha + c \alpha_i$ and  $\gamma= \widetilde \gamma + d \alpha_i$ where $\widetilde \alpha$, $\widetilde \gamma$ are string generators. 
By Lemma \ref{lem:sum-of-roots}, $e_\beta\in \bk[e_\alpha, e_\gamma] \subset [M_{\widetilde \alpha}, M_{\widetilde \gamma}]$.
Next, using the Jacobi identity, we see that for each $\beta + h \alpha_i$ in the string generated by $\beta$, we also have 
\begin{align*}
e_{\beta + h \alpha_i}\in \Bbbk (\ad e_i)^h e_\beta \subset \Bbbk (\ad e_i)^h[e_\alpha, e_\gamma] \subset [M_{\widetilde \alpha}, M_{\widetilde \gamma}],
\end{align*}
since $M_{\widetilde \alpha}$ and $M_{\widetilde \gamma}$  are $\ad e_i$-stable.
This means that $M_\beta \subset [M_{\widetilde \alpha}, M_{\widetilde \gamma}]$, and now the proposition follows by induction and the Jacobi identity.
\end{proof}

Motivated by Propositions \ref{prop:semisimplification-multidimension} and \ref{prop:generation-degree1-alphap}, we describe how the Lie bracket of $(\fg', e_i)$ behaves when restricted to tensor products of some simple submodules.
We describe the adjoint action of an element in $\ker \xi_i$ as a first step.

\begin{lemma}\label{lem:torus-action-alphap}
Let $h\in\ker\xi_i$. If $\beta$ is an $\alpha_i$-string generator, then $h$ acts on $M_\beta$ by 
$\xi_{\beta}(h)$, and on $N_\beta$ by $-\xi_{\beta}(h)$.
\end{lemma}

\begin{proof}
We verify first the claim concerning $M_\beta$. Let $\{\beta+k\alpha_i \colon 0\leq k <d\}$ denote the $\alpha_i$-string through $\beta$. By \eqref{eq:action-h-on-root-vector}, $[h,e_\beta]=\xi_{\beta}(h) e_\beta$.  We show that $[h, e_{\beta+k\alpha_i}]=\xi_{\beta}(h)  e_{\beta+k\alpha_i}$ for all such $k$. For $k=1$, we write $e_{\beta+\alpha_i} = \mu  [e_\beta, e_i]$ and, since $h\in\ker\xi_i$, we get
\begin{align*}
[h, e_{\beta+\alpha_i}]=\mu [[h,e_\beta],e_i] + \mu [e_\beta, [h,e_i]] = \mu \xi_{\beta}(h)  [e_\beta,e_i]=\xi_{\beta}(h) e_{\beta+\alpha_i}.
\end{align*}
The proof now follows inductively on $k$ using that $e_{\beta+k\alpha_i} \in \bk [e_i,e_{\beta+(k-1)\alpha_i}]$. 

The claim regarding the action on $N_\beta$ follows via the Chevalley involution.
\end{proof}

If $M, N$ are submodules of $(\fg', e_i)$, $[M,N]$ denotes the image of  $\restr{[-,-]}{M \otimes N}\colon M \otimes N \to \fg'$. 

In the next couple of Lemmas, we describe $[-,-]\in \Hom_{\balp_p} ( M_{\alpha_j}\otimes N_{\alpha_k}, \fg')$ for $j$, $k\in \bI_\theta$, both different from $i$. The case $j\ne k$ is straightforward:

\begin{lemma}\label{lem:bracket-different-alphap}
Let $j, k\in \bI$, both different from $i$. If $j\ne k$, then $[M_{\alpha_j}, N_{\alpha_k}]=0$.
\end{lemma}

\begin{proof}
If $n, m\geq 0$, then $[(\ad e_i)^n e_j, (\ad f_i)^m f_k]=0$ because $(n-m)\alpha_i+\alpha_j-\alpha_k$ can not be a root.
\end{proof}

On the other hand, when $j=k$ a more careful analysis is needed.

\begin{notation}
If $a\in \bk$ and $n\in\bN$ we denote $(n)_a=n+a$ and  $(n)^!_a=(n)_a(n-1)_a\cdots(1)_a$.
\end{notation}
\begin{remark}\label{rem:bracket-same-alphap}
Assume that $a_{ii}=2$. Let $j\in \bI_\theta$ different from $i$. Then the bracket in $\fg(A)$ satisfies:
\begin{center}
\begin{tabular}{c|c c c c}
$[-,-]$    & $f_{j}$       & $f_{ij}$      & $f_{iij}$     & $f_{iiij}$\\
\hline
 $e_{j}$   & $h_j$         & $a_{ji} f_i$   & 0            &       0\\
 $e_{ij}$  &  $-a_{ji}e_i$ &$a_{ij}h_j+a_{ji}h_i$ & $2a_{ji}(1)_{a_{ij}}f_i$ & 0\\
 $e_{iij}$  & 0     &  $-2a_{ji}(1)_{a_{ij}}e_i$ &  $2(1)_{a_{ij}} (a_{ij}h_j+2a_{ji}h_i)$ & $6a_{ji}(2)^!_{a_{ij}}f_i$\\  
 $e_{iiij}$ &  0 & 0 & $-6a_{ji}(2)^!_{a_{ij}}e_i$ & $6(2)^!_{a_{ij}} (a_{ij}h_j+3a_{ji}h_i)$
\end{tabular}
\end{center}
\end{remark}

We fix cyclic bases for the indecomposable submodules appearing in \Cref{prop:semisimplification-multidimension} \ref{item:semisimplification-multidimension-ii}.

\begin{notation}\label{notation:indecomposables-bases}
Assume $\fg$ is finite dimensional. Let $i\in\bI_\theta$ be such that $a_{ii}=2$. Let $j\in\bI_\theta$ be different from $i$. Then
\begin{itemize}[leftmargin=*]
\item $M_{\alpha_j}=\langle e_j\rangle = \bk\{e_j, (\ad e_i)e_{j}, \dots,  (\ad e_i)^{-c_{ij}}e_{j}\}$.
\item $N_{\alpha_j}=\langle (\ad f_i)^{-c_{ij}}f_{j}\rangle =\bk\{(\ad f_i)^{-c_{ij}}f_{j}, (\ad f_i)^{1-c_{ij}}f_{j}, \dots,f_{j} \}$.
\item $S$ denotes the copy of $L_3$ generated by $e_i$, which has a cyclic basis $\{f_i, h_i, -2e_i\}$.
\item Write $\widetilde{h}_j= h_j$ if $a_{ij}=0$, and $\widetilde{h}_j=2h_j-a_{ji}h_i$ if $a_{ij}\ne 0$. Then $\widetilde h_j\in\ker \xi_i$ and generates a copy of $L_1$.
\item For any direct summand $X$ of $\fg'$, denote by $\iota_{X} \colon X \hookrightarrow \fg'$ the canonical inclusion. If $Y$ is another object we also denote by $\iota_{X}$ the composition $X\oplus Y\twoheadrightarrow X \hookrightarrow \fg'$.
\end{itemize}
\end{notation}

To obtain an explicit expression for the cyclic basis of $N_{\alpha_j}$, one uses the following:
\begin{remark}\label{rem:bracket-cyclic-N}
Assume that $a_{ii}=2$. Let $j\in \bI$ different from $i$. Then
\begin{align*}
[e_i, f_{ij}]&=-a_{ij}f_j, & [e_i, f_{iij}]&=-2(1)_{a_{ij}}f_{ij}, & [e_i, f_{iiij}]&=-3(2)_{a_{ij}}f_{iij}.
\end{align*}
\end{remark}

Now we can use these bases to explicitly describe $[-,-]\in \Hom_{\balp_p} ( M_{\alpha_j}\otimes N_{\alpha_j}, \fg')$.

\begin{lemma}\label{lem:bracket-same-alphap}
Assume $\fg$ is finite dimensional and $a_{ii}=2$. Let $j\in \bI$ different from $i$ and assume that $a_{ij}\in \bF_p$. Consider on $M_{\alpha_j}$ and $N_{\alpha_j}$ the cyclic bases from \Cref{notation:indecomposables-bases}.
\begin{enumerate}[leftmargin=*,label=\rm{(\alph*)}]
\item \label{item:bracket-same-alphap-aij=0} If $c_{ij}=0$, then $[-,-] \colon M_{\alpha_j}\otimes N_{\alpha_j} \to \bk \widetilde{h}_j$ is the canonical isomorphism $L_1\otimes L_1 \to L_1$.
\end{enumerate}
For $-1\geq c_{ij} \geq -3$ we have $[M_{\alpha_j},N_{\alpha_j}]=\bk \widetilde{h}_j \oplus S=L_1\oplus L_3$. Also
\begin{enumerate}[leftmargin=*,label=\rm{(\alph*)}, resume]
\item \label{item:bracket-same-alphap-aij=-1} Assume $c_{ij}=-1$ and $p>2$. If we identify $M_{\alpha_j} \otimes N_{\alpha_j}=L_1\oplus L_3$ as in \Cref{ex:alphap-isotypic-L2}, then $[-,-] \colon M_{\alpha_j}\otimes N_{\alpha_j} \to  \bk \widetilde{h}_j \oplus S$ acts as $\iota_{L_1} \oplus a_{ji} \iota_{L_3}$. 

\item \label{item:bracket-same-alphap-aij=-2} Assume $c_{ij}=-2$ and $p>3$. If we identify $M_{\alpha_j} \otimes N_{\alpha_j}=L_1\oplus L_3\oplus L_5$ as in \Cref{ex:alphap-isotypic-L3-p>3}, then $[-,-] \colon M_{\alpha_j}\otimes N_{\alpha_j} \to  \bk \widetilde{h}_j \oplus S$ acts as $6\iota_{L_1} \oplus 4a_{ji} \iota_{L_3}$. 

\item \label{item:bracket-same-alphap-aij=-3-p=5} Assume $c_{ij}=-3$ and $p=5$. If we identify $M_{\alpha_j} \otimes N_{\alpha_j}=L_1\oplus 3L_5$ as in \Cref{ex:alphap-isotypic-L4-p>5}, then $[-,-] \colon M_{\alpha_j}\otimes N_{\alpha_j} \to  \bk \widetilde{h}_j \oplus S$ restricted to $L_1$ acts as $72\iota_{L_1}$. 

\item \label{item:bracket-same-alphap-aij=-3-p>5} Assume $c_{ij}=-3$ and $p>5$. If we identify $M_{\alpha_j} \otimes N_{\alpha_j}=L_1\oplus L_3\oplus L_5\oplus L_7$ as in \Cref{ex:alphap-isotypic-L4-p>5}, then $[-,-] \colon M_{\alpha_j}\otimes N_{\alpha_j} \to \bk \widetilde{h}_j \oplus S$ acts as $72\iota_{L_1} \oplus 120 a_{ji} \iota_{L_3}$. 
\end{enumerate}
\end{lemma}

\begin{proof}
\ref{item:bracket-same-alphap-aij=0} In this case $M_{\alpha_j}=\bk\{e_j\}$, $N_{\alpha_j}=\bk\{f_j\}$. Since $[e_j,f_j]=h_j$, the claim follows.

The inclusion $[M_{\alpha_j},N_{\alpha_j}]\subseteq\bk \{2h_j-a_{ji}h_i\} \oplus \bk\{f_i, h_i, -2e_i\}=L_1\oplus L_3$ follows directly from  \Cref{rem:bracket-same-alphap}.

\ref{item:bracket-same-alphap-aij=-1}  The cyclic bases for $M_{\alpha_j}$ and $N_{\alpha_j}$ given by \Cref{notation:indecomposables-bases} are $M_{\alpha_j}= \bk \{e_j, e_{ij}\}$ and $N_{\alpha_j}=\bk\{f_{ij}, f_j\}$. According to \Cref{ex:alphap-isotypic-L2}, we have $M_{\alpha_j} \otimes N_{\alpha_j}=L_1\oplus L_3$ with cyclic bases $L_1=\bk\{e_j\otimes f_j - e_{ij}\otimes f_{ij}\}$ and $L_3=\bk\{e_j\otimes f_{ij}, e_{ij}\otimes f_{ij}+ e_j\otimes f_j, 2e_{ij}\otimes f_j\}$. Now we use \Cref{rem:bracket-same-alphap} to compute the action of the Lie bracket on these bases
\begin{align*}
&\begin{aligned}
[e_j,  f_j ]- [e_{ij},f_{ij}]=h_j-a_{ij}h_j-a_{ji}h_i=2h_j-a_{ji}h_i;
\end{aligned}\\
&\begin{aligned}
[e_{j},f_{ij}]&=a_{ji}f_i, &
[e_{ij},f_{ij}]+[e_j,  f_j ]&= a_{ij}h_j+a_{ji}+h_j=a_{ji}h_i, &
2[e_{ij},f_j]&=-2a_{ji}e_i,
\end{aligned}\end{align*}
and the claim follows.

\ref{item:bracket-same-alphap-aij=-2} The cyclic bases for $M_{\alpha_j}$ and $N_{\alpha_j}$ given by \Cref{notation:indecomposables-bases} are $M_{\alpha_j}= \bk \{e_j, e_{ij}, e_{iij}\}$ and $N_{\alpha_j}=\bk \{f_{iij}, 2f_{ij}, 4f_j\}$. According to \Cref{ex:alphap-isotypic-L3-p>3} we have $M_{\alpha_j} \otimes N_{\alpha_j}=L_1\oplus L_3\oplus L_5$ with cyclic generators
\begin{align*}
L_1&=\langle 4e_j\otimes f_j - 2 e_{ij} \otimes f_{ij}+ e_{iij}\otimes f_{iij}\rangle, &
L_3&=\langle 2e_j\otimes f_{ij} - e_{ij}\otimes f_{iij}\rangle, &
L_5&= \langle  e_j\otimes f_{iij} \rangle.
\end{align*}
Next, we use \Cref{rem:bracket-same-alphap} to compute the Lie bracket on these generators
\begin{align*}
4[e_j, f_j] - 2 [e_{ij}, f_{ij}]+ [e_{iij},f_{iij}]=4h_j +4h_j-2a_{ji}h_i+4h_j-4a_{ji}h_i=6(2h_j-a_{ji}h_i),
\end{align*}
thus the bracket acts as $6 \iota_{L_1}$ on $L_1$. Similarly
\begin{align*}
2[e_j, f_{ij}] - [e_{ij},f_{iij}]&= (2a_{ji}+2a_{ji}) f_i=4a_{ji} f_i
\end{align*}
so the bracket acts as $4a_{ji}$ on the fixed basis of $L_3$. Finally
\begin{align*}
[e_j,f_{iij}]&=0.
\end{align*}
Thus $L_5$ is annihilated by the Lie bracket, and the claim follows.

\ref{item:bracket-same-alphap-aij=-3-p>5} The indecomposables $M_{\alpha_j}$ and $N_{\alpha_j}$ have cyclic  bases $M_{\alpha_j}= \bk \{e_j, e_{ij}, e_{iij}, e_{iiij}\}$ and $N_{\alpha_j}=\bk \{f_{iiij}, 3f_{iij}, 12f_{ij}, 36f_j\}$. According to \Cref{ex:alphap-isotypic-L4-p>5}, we can decompose $M_{\alpha_j} \otimes N_{\alpha_j}=L_1\oplus L_3\oplus L_5\oplus L_7$ with cyclic generators
\begin{align*}
L_1&=\langle 36e_j\otimes f_j - 12 e_{ij} \otimes f_{ij}+ 3e_{iij}\otimes f_{iij}-e_{iiij}\otimes f_{iiij}\rangle,\\
L_3&=\langle 36e_j\otimes f_{ij} - 12e_{ij}\otimes f_{iij}+3e_{iij}\otimes f_{iiij}\rangle, \\
L_5&=\langle 3 e_j\otimes f_{iij}-e_{ij}\otimes f_{iiij}\rangle, \\
L_7&= \langle e_j\otimes f_{iiij}\rangle.
\end{align*}
Using \Cref{rem:bracket-same-alphap}, we compute the Lie bracket on the generators. In $L_1$ we have
\begin{align*}
36&[e_j, f_j] - 12 [e_{ij}, f_{ij}]+3 [e_{iij},f_{iij}]-[e_{iiij},f_{iiij}]=\\
&=36h_j -12(-3h_j+a_{ji}h_i)-12 (-3h_j+2a_{ji}h_i)-12(-3h_j+3a_{ji}h_i)=72(2h_j-a_{ji}h_i),
\end{align*}
hence $[-,-]$ acts as $72\iota$ on $L_1$ (note that this argument also works for \ref{item:bracket-same-alphap-aij=-3-p=5}). Similarly, the action on the generator of $L_3$ is
\begin{align*}
36[e_j, f_{ij}] - 12[e_{ij},f_{iij}]+3[e_{iij},f_{iiij}]=(36 +48+36) a_{ji} f_i=120 a_{ji}f_i.
\end{align*}
And the generators of $L_5$ and $L_7$ are mapped to zero.
\end{proof}

The assumption $a_{ij}\in \bF_p$ in \Cref{lem:bracket-same-alphap} assures that $a_{ij}$ is the class of $c_{ij}$ modulo $p$, see \eqref{eq:definition C}, and is satisfied for all matrices $A$ such that $\dim\fg(A)<\infty$ except the Brown algebra in \Cref{ex:brown}. Thus we will avoid the case when $a_{ij}$ is not in the prime field.

\smallbreak
Next, we describe the adjoint action of $S=\{f_i, h_i, -2e_i\}\simeq L_3$ on itself and on the submodules $\widetilde{h}_j$, $M_{\alpha_j}$, $N_{\alpha_j}$ from \Cref{notation:indecomposables-bases}. 
We will need the following computations:

\begin{remark}\label{rem:bracket-S-M-alphap}
Assume that $a_{ii}=2$. Let $j\in \bI$ different from $i$. Then
\begin{align*}
[f_i, e_{j}]&=0, &[f_i, e_{ij}]&=-a_{ij}e_j, & [f_i, e_{iij}]&=-2(1)_{a_{ij}}e_{ij}, & [f_i, e_{iiij}]&=-3(2)_{a_{ij}}e_{iij}.
\end{align*}
\end{remark}

\begin{lemma}\label{lem:bracket-sl2-g-alphap}
Assume $\fg$ is finite dimensional and $a_{ii}=2$. Let $j\in \bI$ different from $i$ and assume that $a_{ij}\in \bF_p$. Consider on $S$, $M_{\alpha_j}$, and $N_{\alpha_j}$ the cyclic bases from \Cref{notation:indecomposables-bases}.
\begin{enumerate}[leftmargin=*,label=\rm{(\alph*)}]
\item \label{item:bracket-sl2-sl2-alphap} We have $[S,S]=S$. If $p>3$ and we identify $S\otimes S=L_1\oplus L_3\oplus L_5$ as in \Cref{ex:alphap-isotypic-L3-p>3}, then  $[-,-] \colon S\otimes S \to S$ acts as $4\iota_{L_3}$.
\item \label{item:bracket-sl2-hj-alphap} We have $[S,\widetilde{h}_j]=0$.
\item \label{item:bracket-sl2-Mj-alphap-aij=0} If $c_{ij}=0$, then $[S,M_{\alpha_j}]=[S,N_{\alpha_j}]=0$.
\end{enumerate}
For $-1\geq c_{ij} \geq -3$, we have $[S,M_{\alpha_j}]=M_{\alpha_j}$ and $[S,N_{\alpha_j}]=N_{\alpha_j}$. Also
\begin{enumerate}[leftmargin=*,label=\rm{(\alph*)}, resume]
\item \label{item:bracket-sl2-Mj-alphap-aij=-1} Assume $c_{ij}=-1$ and $p>3$. If we identify $S\otimes M_{\alpha_j}=L_2\oplus L_4$ as in \Cref{ex:alphap-isotypic-L3-L2-p>3}, then $[-,-] \colon S\otimes M_{\alpha_j} \to  M_{\alpha_j}$ acts as $3\iota_{L_2}$. 

\item \label{item:bracket-sl2-Mj-alphap-aij=-2} Assume $c_{ij}=-2$ and $p>3$. If we identify $S\otimes M_{\alpha_j}=L_1\oplus L_3\oplus L_5$ as in \Cref{ex:alphap-isotypic-L3-p>3}, then $[-,-] \colon S\otimes M_{\alpha_j} \to M_{\alpha_j}$ acts as $4\iota_{L_3}$. 

\item \label{item:bracket-sl2-Mj-alphap-aij=-3-p=5} Assume $c_{ij}=-3$ and $p=5$. If we identify $S\otimes M_{\alpha_j}=L_2\oplus 2L_5$ as in \Cref{ex:alphap-isotypic-L3-L4-p>5}, then $[-,-] \colon S\otimes M_{\alpha_j} \to M_{\alpha_j}$ vanishes $L_2$.

\item \label{item:bracket-sl2-Mj-alphap-aij=-3-p>5} Assume $c_{ij}=-3$ and $p>5$. If we identify $S\otimes M_{\alpha_j}=L_2\oplus L_4\oplus L_6$ as in \Cref{ex:alphap-isotypic-L3-L4-p>5}, then $[-,-] \colon S\otimes M_{\alpha_j} \to M_{\alpha_j}$ acts as $15\iota_{L_4}$.
\end{enumerate}
In \ref{item:bracket-sl2-Mj-alphap-aij=-1}-\ref{item:bracket-sl2-Mj-alphap-aij=-3-p>5}, the exact same argument holds for $N_{\alpha_j}$ in place of  $M_{\alpha_j}$ (with the same scalars).
\end{lemma}

\begin{proof}
\ref{item:bracket-sl2-sl2-alphap} According to \Cref{ex:alphap-isotypic-L3-p>3}, we have $S\otimes S=L_1\oplus L_3\oplus L_5$ with cyclic generators
\begin{align*}
L_1&=\langle -2f_i\otimes e_i-h_i\otimes h_i-2e_i\otimes f_i\rangle, &
L_3&=\langle f_i\otimes h_i- h_i\otimes f_i\rangle, &
L_5&= \langle  f_i\otimes f_i \rangle.
\end{align*}
Now, it is easy to see that $[-,-]$ annihilates the generators of $L_1$ and $L_5$, and maps the generator of $L_3$ to $4f_i$.

\ref{item:bracket-sl2-hj-alphap} Follows immediately because $S$ is generated by $f_j$, and $\widetilde{h}_j$ is in $\ker(\xi_i)$ by construction. 

\ref{item:bracket-sl2-Mj-alphap-aij=0} In this case, $M_{\alpha_j}=\bk \{e_j\}$. Since $f_i$, $h_i$, and $e_i$ annihilates $e_j$, we have $[S,M_{\alpha_j}]=0$. 

\ref{item:bracket-sl2-Mj-alphap-aij=-1} According to \Cref{ex:alphap-isotypic-L3-L2-p>3}, we have $S\otimes M_{\alpha_j}= L_2\oplus L_4$, where 
\begin{align*}
L_2&=\langle 2f_i\otimes e_{ij} -h_i\otimes e_j \rangle, &
L_4&=\langle f_i\otimes e_j\rangle.
\end{align*}
Now, we use \Cref{rem:bracket-S-M-alphap} to see that $[-,-]$ annihilates the generator of $L_4$ and maps that of $L_2$ to $3 e_j$.

\ref{item:bracket-sl2-Mj-alphap-aij=-2} From \Cref{ex:alphap-isotypic-L3-p>3}, we get $S\otimes M_{\alpha_j}=L_1\oplus L_3\oplus L_5$, with cyclic generators
\begin{align*}
L_1&=\langle f_i\otimes e_{iij}-h_i\otimes e_{ij}-2e_i\otimes e_i\rangle, &
L_3&=\langle f_i\otimes e_{ij}- h_i\otimes e_i\rangle, &
L_5&= \langle  f_i\otimes e_j \rangle.
\end{align*}
Now the proof follows as in part \ref{item:bracket-sl2-sl2-alphap}.

\ref{item:bracket-sl2-Mj-alphap-aij=-3-p>5} We use \Cref{ex:alphap-isotypic-L3-p>3} to get $S\otimes L_3=L_1\oplus L_3\oplus L_5$ with cyclic generators 
\begin{align*}
L_2&=\langle -2f_i\otimes e_i-h_i\otimes h_i-2e_i\otimes f_i\rangle, &
L_4&=\langle f_i\otimes h_i- h_i\otimes f_i\rangle, &
L_6&= \langle  f_i\otimes e_i \rangle.
\end{align*}
An straightforward computation using \Cref{rem:bracket-S-M-alphap} shows that $[-,-]$ annihilates the generators of $L_2$ and $L_6$, and it maps the generator of $L_4$ to $15 e_j$. A similar argument works for \ref{item:bracket-sl2-Mj-alphap-aij=-3-p=5}.
\end{proof}

\subsection{The structure of the Lie algebra \texorpdfstring{$\sfS (\fg'(A), e_i)$}{}}\label{subsec:g-in-verp}
We work towards a root system associated to the semisimplification $\sfS (\fg'(A), e_i)$, where $\fg(A)$ is a finite dimensional contragredient Lie algebra, and $i\in \bI_\theta$ is such that $a_{ii}=2$. Recall the sets $\varDelta_{+,\min}^{A}$, $\varDelta_{\min}^{A}$ from \Cref{notation:strings}.

\smallbreak
Let $\pi_{i}\colon\bZ^{\theta}\to\bZ^{\theta-1}$ be the projection introduced in Lemma \ref{lem:Cuntz-Lentner}. Consider
\begin{align*}
\nabla_{\pm}^{\sfS(\fg, e_i)} &\coloneqq \pi_{i}\left( \varDelta^{A}_{\pm,\min} \right), & \nabla^{\sfS(\fg, e_i)} &\coloneqq\pi_{i}\left( \varDelta^{A}_{\min} \right)=\nabla_{+}^{\sfS(\fg, e_i)}\cup \nabla_{-}^{\sfS(\fg, e_i)}.
\end{align*}

We will see that $\nabla^{\sfS(\fg, e_i)}$ may contain positive multiples of roots. This behavior resembles that of contragredient Lie superalgebras, as described in \Cref{subsec:groupoid}. Hence we introduce
\begin{align*}
\varDelta_{\pm}^{\sfS(\fg, e_i)} & \coloneqq \nabla_{\pm}^{\sfS(\fg, e_i)}-\{n\beta: n\ge 2, \beta \in \nabla_{\pm}^{\sfS(\fg, e_i)}\}, & 
\varDelta^{\sfS(\fg, e_i)} &\coloneqq \varDelta_{+}^{\sfS(\fg, e_i)}\cup \varDelta_{-}^{\sfS(\fg, e_i)}.
\end{align*}

\begin{remark} Recall the notations $\pi_i$ and $\varDelta_{\pm,\min}^{A}$ from the beginning of \Cref{subsec:g-in-verp}.
\begin{enumerate}[leftmargin=*,label=\rm{(\roman*)}]
\item Via $\pi_i\colon \bZ^\theta \to\bZ^{\theta-1}$, we obtain a grading of $(\fg, e_i)$ by $\bZ^{\theta-1}$. Explicitly, 
\begin{align*}
\fg&=\bigoplus_{\gamma\in \bZ^{\theta-1}} \fg_{\gamma}, & \text{where } \fg_{\gamma}&\coloneqq\bigoplus_{\beta \in \pi_i^{-1}(\gamma)} \fg_{\beta}.
\end{align*}
\item For each $\beta\in\varDelta_{+,\min}^{A}$, we have $M_{\beta}\subseteq \fg_{\pi_i(\beta)}$ and $N_{\beta}\subseteq \fg_{-\pi_i(\beta)}$.
Thus $\sfS(\fg',e_i)$ inherits the $\bZ^{\theta-1}$-graduation of $\fg'$. Moreover,
\begin{align*}
\sfS(\fg',e_i) &= \sfS(\fg',e_i)_0 \oplus \left( \bigoplus_{\gamma\in\nabla^{\sfS(\fg, e_i)}} \sfS(\fg, e_i)_{\gamma} \right).
\end{align*}
\end{enumerate}
\end{remark}

To illustrate the situation, we explicitly compute $\nabla^{\sfS(\fg, e_i)}$ and $\varDelta^{\sfS(\fg, e_i)}$ for all finite Cartan matrices of rank $2$. We show that, in all cases, $\varDelta^{\sfS(\fg, e_i)}=\{\pm\alpha_j\}$ where $j$ is the index different from $i$; however $\nabla^{\sfS(\fg, e_i)}$ depends on the number of laces of the Dynkin diagram.

\begin{example}\label{ex:A2}
Let $\fg$ be of type $A_2$; that is, $A=\left(\begin{smallmatrix} 2 & -1 \\ -1 & 2\end{smallmatrix}\right)$. If we fix $i=1$, then $\nabla^{\sfS(\fg, e_1)}_+=\{2\}$. Indeed, $\fg = M_{2}\oplus S \oplus\Bbbk\{2h_2+h_1\} \oplus N_{2}$, where $M_{2}=\Bbbk\{e_2,e_{12}\}\simeq L_2$, $N_{2}=\Bbbk\{f_{12}, f_2\} \simeq L_2$, and one can easily verify using the bases from \Cref{ex:alphap-isotypic-L2} that $\restr{[-,-]}{M_{2}\ot M_{2}}=0$.
\end{example}

\begin{example}\label{ex:brown}
Let $p=3$, $A=\left(\begin{smallmatrix} 2 & a \\ -1 & 2 \end{smallmatrix}\right)$, $a\notin \bF_3$. Thus $\fg$ is the Brown Lie algebra $\brown(2,a)$.

\begin{itemize}[leftmargin=*]
\item For $i=1$, $\nabla^{\sfS(\fg, e_1)}_+=\emptyset$ since 
$\fg = M_{2}\oplus S \oplus\Bbbk\{2h_2+h_1\} \oplus N_{2}$, where $M_{2}=\Bbbk\{e_2,e_{12},e_{112}\}$, $N_{2}=\Bbbk\{f_{112},2f_{12},4f_{2}\}$, and $M_{2},N_{2}\simeq L_3$.
Here we have $\restr{[-,-]}{M_{2}\ot M_{2}}=0$.

\item For $i=2$, we obtain $\nabla^{\sfS(\fg, e_2)}_+=\{1,1^2\}$ since
\begin{align*}
(\fg, e_2) &= M_{1} \oplus M_{1^22}\oplus S \oplus\Bbbk\{2h_1+2h_2\} \oplus N_{1} \oplus N_{1^22},
\end{align*}
where $M_{1}=\Bbbk\{e_1,e_{12}\}$, $M_{1^22}=\Bbbk\{e_{112}\}$, $N_{1}=\Bbbk\{f_{12},f_{1}\}$, $N_{1^22}=\Bbbk\{f_{112}\}$. 

According to \Cref{ex:alphap-isotypic-L2}, we have $M_{1}\ot M_{1}\simeq L_1\oplus L_3$, where the copy of $L_1$ has basis
$e_1\ot e_{12}-e_{12}\ot e_1$ so $[-,-]\colon M_{1}\ot M_{1}\to M_{1^22}\simeq L_1$ is $2\iota_{L_1}$. It is straightforward to see that $\restr{[-,-]}{M_{1^22}\ot M_{2}}=0$ and $\restr{[-,-]}{M_{1^22}\ot M_{1^22}}=0$.
\end{itemize}
\end{example}

\begin{example}\label{ex:B2}
Let $\fg$ be of type $B_2$; that is, $A=\left(\begin{smallmatrix} 2 & -2 \\ -1 & 2\end{smallmatrix}\right)$.  In this case,
\begin{align*}
\nabla^{\sfS(\fg, e_1)}_+ & =\begin{cases} \{2\}, & p>3, \\ \emptyset, & p=3; \end{cases} &
\nabla^{\sfS(\fg, e_2)}_+&=\{1,1^2\}.
\end{align*}
The proof and the description of the brackets are similar to those of  \Cref{ex:brown}.
\end{example}

\begin{example}\label{ex:G2}
Assume $p>3$ and consider $\fg$ of type $G_2$, that is, $A=\left(\begin{smallmatrix} 2 & -3 \\ -1 & 2\end{smallmatrix}\right)$. 

\begin{itemize}[leftmargin=*]
\item For $i=1$, we obtain $\nabla^{\sfS(\fg, e_1)}_+=\{2,2^2\}$. Indeed, 
\begin{align*}
(\fg, e_1) &= M_{2} \oplus M_{1^32^2}\oplus S \oplus\Bbbk\{h_1+2h_2\} \oplus N_{2} \oplus N_{1^32^2},
\\
&\begin{aligned}
\text{where }M_{2}&=\Bbbk\{e_2,e_{12},e_{112},e_{1112}\}, & M_{1^32^2}&=\Bbbk\{[e_{112},e_{12}]\}, 
\\
N_{2}&=\Bbbk\{f_{1112},3f_{112},12f_{12},36f_{2}\}, & N_{1^32^2}&=\Bbbk\{[f_{112},f_{12}]\}.
\end{aligned}
\end{align*}
Here $M_{2}\ot M_{2}\simeq L_1\oplus L_3\oplus L_5\oplus L_7$ when $p>5$, while for $p=5$ we have $M_{2}\ot M_{2}\simeq L_1\oplus 3 L_5$, see  \Cref{ex:alphap-isotypic-L4-p>5}. In any case, inside $M_{2}\otimes M_2$ there is a copy of $L_1$ spanned by
$e_2\otimes e_{1112}-e_{12}\otimes e_{112}+e_{112}\otimes e_{12}-e_{1112}\otimes e_{2}$. Via $M_{1^32^2}\simeq L_1$, we see that $[-,-]\colon M_{1}\ot M_{1}\to M_{1^32^2}$ is $4\iota_{L_1}$. We also have $\restr{[-,-]}{M_{1^32^2}\ot M_{2}}=0$ and $\restr{[-,-]}{M_{1^32^2}\ot M_{1^32^2}}=0$.

\item If $i=2$, we get $\nabla^{\sfS(\fg, e_2)}_+=\{1,1^2,1^3\}$. In fact, we decompose
\begin{align*}
(\fg, e_2 )&= M_{1} \oplus M_{1^22}\oplus M_{1^32}\oplus S \oplus\Bbbk\{2h_1+3h_2\} \oplus N_{1} \oplus N_{1^22}\oplus N_{1^32},
\end{align*}
where 
\begin{align*}
M_{1}&=\Bbbk\{e_1,e_{12}\}, & M_{1^22}&=\Bbbk\{e_{112}\}, & M_{1^32}&=\Bbbk\{e_{1112}, \lambda_1[e_{112},e_{12}]\}, 
\\
N_{1}&=\Bbbk\{f_{12},f_1\}, & N_{1^22}&=\Bbbk\{f_{112}\}, & N_{1^32}&=\Bbbk\{ [f_{112},f_{12}], \lambda_2f_{1112}\}. 
\end{align*}
Here $M_{1}\ot M_{1}\simeq L_1\oplus L_3$, $M_{1^22}\simeq L_1$, and $M_{1^32}\simeq L_2$. Using \Cref{ex:alphap-isotypic-L2} we see that $[-,-]\colon M_{1}\ot M_{1}\to M_{1^22}$ is $2\iota_{L_1}$. Also, $[-,-]\colon M_{1}\ot M_{1^22}\to M_{1^32}$ is $4\iota_{L_2}$ and the other brackets between $M_{\alpha}$'s are 0.
\end{itemize}
\end{example}

We give another example, this time for a matrix of rank three.

\begin{example}\label{ex:br(3)}
Let $p=3$. We describe \emph{root vectors}  for semisimplifications of the Lie algebra $\brown(3)$ of dimension $29$. 
There are two matrices realizing this Lie algebra, cf. \cite{Sk}:
\begin{align}\label{eq:matrices br(3)}
A_1&=\begin{bmatrix} 2 & -1 & 0 \\ -1 & 2 & -1 \\ 0 & 1 & 0 \end{bmatrix}, &
A_2&=\begin{bmatrix} 2 & -1 & 0 \\ -2 & 2 & -1 \\ 0 & 1 & 0 \end{bmatrix}.
\end{align}
The corresponding positive roots are
\begin{align*}
\varDelta_{+}^{A_1}= & \{ 1, 12, 123,  1^22^33^4, 12^23^2, 12^23^3, 12^23^4, 12^33^4, 123^2,  2, 23^2, 23, 3\}, 
\\
\varDelta_{+}^{A_2}= & \{ 1, 12^2, 12, 123^2, 12^33^2, 1^22^33^2, 12^23^2, 123, 12^23, 2, 23^2, 23, 3 \}.
\end{align*}
As $a_{33}=b_{33}=0$, we consider semisimplifications by $e_i$ with $i\ne 3$.  For $\fg=\fg(A_1)$ we have:
\begin{align*}
\nabla^{\sfS(\fg, e_1)}_+&=\{2, 23, 23^2, 2^23^2, 2^23^3, 2^23^4, 2^33^4, 3\}, &
\nabla^{\sfS(\fg, e_2)}_+&=\{1, 13,  1^23^4, 13^2, 13^3, 13^4, 3^2, 3\}.
\end{align*}
And for $\fg=\fg(A_2)$ we have:
\begin{align*}
\nabla^{\sfS(\fg, e_1)}_+&=\{2, 2^2, 2^33^2, 2^23^2, 2^23, 23^2, 23, 3 \}, &
\nabla^{\sfS(\fg, e_2)}_+&=\{13, 1^23^2, 3^2, 3\}.
\end{align*}
We show details on the computations only for the last case. Notice that
\begin{align*}
(\fg,e_2) &= M_{1} \oplus M_{123}\oplus M_{123^2}\oplus M_{1^22^33^2} \oplus M_{23^2}\oplus M_{3} \oplus S \oplus\Bbbk\{2h_1+h_2\} 
\\
& \quad \oplus\Bbbk\{2h_3-h_2\} \oplus N_{1} \oplus N_{123}\oplus N_{123^2}\oplus N_{1^22^33^2} \oplus N_{23^2}\oplus N_{3},
\end{align*}
where 
\begin{align*}
M_{1}&=\Bbbk\{e_1,e_{12},e_{12^2}\}, & M_{123}&=\Bbbk\{e_{123},e_{2123}\}, & M_{123^2}&=\Bbbk\{e_{3123}, e_{23123},e_{223123}\}, 
\\
M_{1^22^33^2}&=\Bbbk\{[e_{2123},e_{123}]\}, & M_{23^2}&=\Bbbk\{e_{332}\}, & M_{3}&=\Bbbk\{ e_3, e_{23} \}.
\end{align*}
and the $N_{\alpha}$'s have analogous descriptions.
Here $M_1, M_{123^2}\simeq L_3$, so these vanish under the semisimplification functor. Also, $M_{123}\ot M_{123}\simeq L_1\oplus L_3 \simeq M_{3}\ot M_{3}$. One can see that
$[-,-]\colon M_{123}\ot M_{123}\to M_{1^22^33^2}$ is $\iota_{L_2}$, $[-,-]\colon M_{3}\ot M_{3}\to M_{23^2}$ is $\iota_{L_2}$, and
the other brackets between $M_{\alpha}$'s are 0.
\end{example}

\begin{proposition}\label{prop:nabla-delta-parabolic-root-system}
Let $A$ be such that $\dim \fg(A)<\infty$ and $i\in\bI_{\theta}$ such that $a_{ii}=2$.
\begin{enumerate}[leftmargin=*,label=\rm{(\roman*)}]
\item\label{item:nabla-delta-simply-laced} If the Dynkin diagram of $A$ is simply laced, then  $\nabla^{\sfS(\fg, e_i)}=\varDelta^{\sfS(\fg, e_i)}$.
\item\label{item:nabla-delta-B} If $A$ is of type $B_{\theta}$, then 
$\nabla^{\sfS(\fg, e_i)}=\varDelta^{\sfS(\fg, e_i)}\cup \{2\alpha_{i+1\,\theta}\}$.
\item\label{item:nabla-delta-C} If $A$ is of type $C_{\theta}$, then 
$\nabla^{\sfS(\fg, e_i)}=\begin{cases}
\varDelta^{\sfS(\fg, e_i)}, & i\ne \theta, \\
\varDelta^{\sfS(\fg, e_\theta)}\cup \{2\alpha_{j\,\theta-1} : j\in\bI_{\theta-1}\}, & i=\theta.
\end{cases}$
\item\label{item:nabla-delta-F4} If $A$ is of type $F_4$, then 
$\nabla^{\sfS(\fg, e_i)}=\begin{cases}
\varDelta^{\sfS(\fg, e_i)}, & i=1,2, \\
\varDelta^{\sfS(\fg, e_3)}\cup \{2\beta : \beta=2, 12\}, & i=3, \\
\varDelta^{\sfS(\fg, e_4)}\cup \{2\beta: \beta=23,123,12^23\}, & i=4.
\end{cases}$
\end{enumerate}
\end{proposition}
\begin{proof}
For \ref{item:nabla-delta-simply-laced}, it is enough to show that the entries of $\pi_i(\beta)$ are coprime for every $\beta\in\varDelta^A_+$ different from $\alpha_i$, which we check case-by-case.
In types $A_{\theta}$, $D_{\theta}$ and $E_6$ the roots $\beta\ne\alpha_i$ are either $\alpha_j$ with $j\ne i$ or else contain at least two coordinates equal to $1$, so $\pi_i(\beta)$ keeps at least one coordinate equal to 1. 

The same happens for all roots in types $E_7$ and $E_8$, except for the largest root in type $E_7$ and twelve roots in type $E_8$. However, these exceptional roots contain one coordinate equal to 1, another equal to 2 and another equal to 3.

\smallbreak
For \ref{item:nabla-delta-B}, at least one entry of the vector $\pi_i(\beta)$ equals $1$, unless $\beta=\alpha_{i\theta}+\alpha_{i+1\,\theta}$: in this case, $\pi_i(\beta)=2\alpha_{i+1\,\theta}$, and $\alpha_{i+1\,\theta}\in\varDelta_{+}^{\sfS(\fg, e_i)}$.

\smallbreak
For \ref{item:nabla-delta-C}, again $\pi_i(\beta)$ has at least one entry equal to one unless $i=\theta$ and $\beta$ is of the form $\beta=\alpha_{j\theta}+\alpha_{j\,\theta-1}$, for some $j\in\bI_{\theta-1}$. In that case $\pi_{\theta}(\beta)=2\alpha_{j\,\theta-1}$, and $\alpha_{j\,\theta-1}\in\varDelta_{+}^{\sfS(\fg, e_\theta)}$.

\smallbreak
Finally, we consider \ref{item:nabla-delta-F4}, whence the positive roots are 
\begin{align}\label{eq:F4-pos-roots}
\begin{aligned}
\varDelta_+^A&= \{1, 12, 2, 1^22^23, 12^23, 123, 2^23, 23, 3, 1^22^43^34, 1^22^43^24, 1^22^33^24, 1^22^23^24, 
\\ & \qquad 1^22^234, 12^33^24, 12^23^24, 1^22^43^34^2, 12^234, 1234, 2^23^24, 2^234, 234, 34, 4\}.
\end{aligned}
\end{align}
If $i=1,2$, then $\pi_i(\beta)$ is not a multiple of another root for all $\beta\in \varDelta_+^A$. 
\begin{description}
\item[$i=3$] The pairs $(\beta,\beta')$ such that $\pi_3(\beta)=2\beta'$ are $(1^22^23, 12)$, $(2^23,2)$. Otherwise $\pi_3(\beta)$ is not a multiple of another vector with integer entries.

\item[$i=4$] The pairs $(\beta,\beta')$ such that $\pi_4(\beta)=2\beta'$ are $(1^22^43^24, 12^23)$, $(1^22^23^24,123)$, $(2^23^24,23)$.
\end{description}
And it is clear that for $i=3,4$, the unique $n \geq 2$ such that $\pi_i(\beta)=n\beta'$, for some $\beta,\beta'$ in $\varDelta^A_+$, is $n=2$. 
\end{proof}

Let us push \Cref{notation:indecomposables-bases} to $\Verp$.

\begin{notation} Let $A\in \Bbbk^{\theta\times\theta}$ be such that $\fg= \fg(A)$ is finite dimensional, and fix $i\in\bI_{\theta}$ with $a_{ii}=2$. 
We set
\begin{align*}
\ov{\bI} &:= \{j\in \bI | j\ne i,  1-c_{ij}^A< p \}.
\end{align*}
For $\beta'\in\nabla_{+}^{\sfS(\fg, e_i)}$ and $j\in\ov{\bI}$, consider the following subobjects of $\sfS(\fg, e_i)$:
\begin{itemize}[leftmargin=*]
\item $\ov{e_{\beta'}}\coloneqq\sfS(M_{\beta})$, $\ov{f_{\beta'}}\coloneqq\sfS(N_{\beta})$, where $\beta \in \varDelta_{+,k}$ is such that $\pi_{i}(\beta)=\beta'$; both are isomorphic to $\ttL_{k}$\footnote{Here, $L_p=0$.}. In particular we set $\ov{e_j}\coloneqq\sfS(M_{\alpha_j})$, $\ov{f_j}\coloneqq\sfS(N_{\alpha_j})$; both are isomorphic to $\ttL_{1-c_{ij}^A}$.
\item $\ov{h_j}\coloneqq\sfS\left(\bk \widetilde{h_{j}} \right)$, isomorphic to $\ttL_1$.
\item $\ov{S}\coloneqq\sfS(S)$, isomorphic to $\ttL_3$, or $0$ if $p=3$.
\end{itemize}
\end{notation} 

Next we summarize the data describing the Lie algebras $\sfS(\fg, e_i)$.

\begin{definition}
We say that a root $\beta\in\varDelta_{+,\min}^A$ is $i$-\emph{good} if there exist $\alpha,\gamma \in\varDelta_{+,\min}^A$ and a decomposition $\beta=\widetilde{\alpha}+\widetilde{\gamma}$ with $\widetilde{\alpha}$ and $\widetilde{\gamma}$ in the $i$-strings of $\alpha$ and $\gamma$, respectively.
\end{definition}

In other words, $\beta$ is $i$-good if there exists a decomposition $\beta=\widetilde{\alpha}+\widetilde{\gamma}$ with $\widetilde{\alpha},\widetilde{\gamma}\in\varDelta_{+}^A$ such that $\widetilde{\alpha}$ and $\widetilde{\gamma}$ do not belong to $i$-strings of length $p$. 

The existence of roots which are not $i$-good corresponds to the fact that the Lie algebra $\sfS(\fg, e_i)$ is not generated by $\ov{e_j}$, $\ov{f_j}$, $\ov{h_j}$. But if all roots are $i$-good, then we will see that $\sfS(\fg, e_i)$ is generated by $\ov{e_j}$, $\ov{f_j}$, $\ov{h_j}$.

\begin{example}
Assume that $p=3$. Let $\fg$ be of type $F_4$ and $i=1$; the positive roots are given in \eqref{eq:F4-pos-roots}. Then $\beta=12^33^24$ is not $1$-good since the possible decompositions are
\begin{align*}
&12^234, 23, & &1234, 2^23, 
& &123, 2^234,& &12^23, 234,
& &12, 2^23^24,& &2, 12^23^24.
\end{align*}
All decompositions have a root in a $1$-string of length $3$ since
\begin{align*}
M_{2^23} &= \Bbbk e_{2^23}\oplus\Bbbk e_{12^23}\oplus\Bbbk e_{1^22^23}, & 
M_{2^234} &=\Bbbk e_{2^234} \oplus\Bbbk e_{12^234}\oplus\Bbbk e_{1^22^234}, 
\\
M_{2^23^24} &=\Bbbk e_{2^23^24}\oplus\Bbbk e_{12^23^24}\oplus\Bbbk e_{1^22^23^24}.
\end{align*}
From here we can deduce that $\ov{e_{2^33^24}}$ does not belong to the subalgebra of $\sfS(\fg, e_i)$ generated by $\ov{e_j}$. This example corresponds to $(\star)$ in 
\cite{Kan}*{\S 4.6.5}.
\end{example}

\begin{theorem}\label{thm:ss-description-S(g)}
Let $A$ be such that $\dim \fg(A)<\infty$, $i\in\bI_{\theta}$ such that $a_{ii}=2$.
\begin{enumerate}[leftmargin=*,label=\rm{(\roman*)}]
\item\label{item:triang-decomp-S(g)} $\sfS(\fg, e_i)$ has a triangular decomposition $\sfS(\fg, e_i)=\sfS(\fg, e_i)_+\oplus\sfS(\fg, e_i)_0\oplus\sfS(\fg, e_i)_-$, where
\begin{align*}
\sfS(\fg, e_i)_{\pm}&=\bigoplus_{\beta\in \nabla^{\sfS(\fg, e_i)}} \sfS(\fg, e_i)_{\pm \beta}, & 
\sfS(\fg, e_i)_0&=\ov S \oplus \left( \oplus_{j\ne i} \ov{h_j} \right).
\end{align*}
Moreover the non-trivial homogeneous components are simple.
\item\label{item:grading-S(g)} $\sfS(\fg, e_i)$ is a $\bZ^{\theta-1}$-graded Lie algebra in $\Ver_p$.
\item\label{item:non-deg-form-S(g)} $\sfS(\fg, e_i)$ has an invariant non-degenerate symmetric form $\mathtt{B}$ such that
\begin{align*}
\restr{\mathtt{B}}{\sfS(\fg, e_i)_{\alpha} \ot \sfS(\fg, e_i)_{\beta}} &=0 && \text{ if }\alpha\ne -\beta\in \bZ^{\theta-1}.
\end{align*}
\item\label{item:gen-deg-one} If every $\beta\in\varDelta_{+,\min}^A$ is $i$-good, then $\sfS(\fg, e_i)$ is generated by $\ov{e_j}$, $\ov{f_j}$, $j\in\ov{\bI}$, and $\sfS(\fg, e_i)_0$. 
\item\label{item:delta-parabolic-root-system} If every $\beta\in\varDelta_{+,\min}^A$ is $i$-good, then $\varDelta^{\sfS(\fg, e_i)}$ is the (reduced) root system of the parabolic restriction orthogonal to $\alpha_i$.
\end{enumerate}
\end{theorem}
\begin{proof}
Notice that \ref{item:triang-decomp-S(g)} follows from the triangular decomposition of $\fg$ and Proposition \ref{prop:non degenerate odd}.
\ref{item:grading-S(g)} and \ref{item:non-deg-form-S(g)} follow from the $\bZ^{\theta}$-grading on $\fg$ and  Lemma \ref{lemma:bilinear-forms-ss}.

\medspace

Next we deal with \ref{item:gen-deg-one}. It suffices to prove that $\ov{e_{\beta'}}$ belongs to the subalgebra generated by $\ov{e_j}$, $j\in\ov{\bI}$, for all $\beta'\in\nabla_{+}$, since the same argument shows that $\ov{f_{\beta'}}$ belongs to the subalgebra generated by $\ov{f_j}$, $j\in\ov{\bI}$, for all $\beta'\in\nabla_{+}$.
The proof is by induction on the height of $\beta'$. If $\beta'$ has height one, then $\beta'=\alpha_j$ for some $j\in\ov{\bI}$ and the claim now follows by hypothesis. 
Now we assume that $\beta'$ has height $>1$. Let $\beta\in\varDelta_{+,\min}^A$ be such that $\beta'=\pi_i(\beta)$. As $\beta$ is $i$-good, there exist 
$\alpha,\gamma \in\varDelta_{+,\min}^A$ and a decomposition $\beta=\widetilde{\alpha}+\widetilde{\gamma}$ with $\widetilde{\alpha}$ and $\widetilde{\gamma}$ in the $i$-strings of $\alpha$ and $\gamma$, respectively. Let $\alpha'=\pi_i(\alpha)$, $\gamma'=\pi_i(\gamma)$.

\begin{claim}\label{claim:bracket-not-zero}
The map $[-,-]:\ov{e_{\alpha'}}\ot\ov{e_{\gamma'}}\to\ov{e_{\beta'}}$ is not zero.
\end{claim}

Let $a,b,c$ be the lengths of $i$-strings of $\alpha$, $\beta$ and $\gamma$, respectively.
Let $j,k\ge 0$ be such that $\widetilde{\alpha}=\alpha+j\alpha_i$ and $\widetilde{\gamma}=\gamma+k\alpha_i$: we may assume that $j\le k$ up to exchange $\alpha$ and $\gamma$. We look for the possible $5$-tuples $(a,b,c,j,k)$. By \Cref{thm:cuntz-heck}, there exists an element $w$ of the Weyl group(oid)  and $i_1,i_2,i_3\in\bI$ such that $w(\alpha_i)=\alpha_{i_1}$, $w(\alpha)\in\bN_0\alpha_{i_1}+\bN_0\alpha_{i_2}$, $w(\gamma)\in\bN_0\alpha_{i_1}+\bN_0\alpha_{i_2}+\bN_0\alpha_{i_3}$: hence, 
$w(\beta)=w(\alpha)+w(\gamma)+(j+k)\alpha_{i_1}\in \bN_0\alpha_{i_1}+\bN_0\alpha_{i_2}+\bN_0\alpha_{i_3}$ and $w$ sends the $i$-strings of $\alpha$, $\beta$, $\gamma$ bijectively to the $i_1$-strings of $w(\alpha)$, $w(\beta)$, $w(\gamma)$, respectively. Thus we have to look for the possible $5$-tuples just in rank 3. If the submatrix of rank 3 is not connected, then the result is straightforward. For connected submatrices of rank 3 we obtain the following:
\begin{itemize}[leftmargin=*]
\item $(1,1,1,0,0)$, for some triples of roots in type $C_3$ ($i=3$), and also in type $\brown(3)$ ($i=1,2$);
\item $(1,2,2,0,0)$ or $(2,2,1,0,0)$, for some triples of roots in types $A_3$ ($i=1$), $B_3$ ($i=1,2$), $C_3$ ($i=1,3$), and both matrices of type $\brown(3)$ ($i=1,2$);
\item $(2,1,2,0,1)$, for some triples of roots in types $A_3$ ($i=2$), $B_3$ ($i=1,2$), $C_3$ ($i=2,3$),and both matrices of type $\brown(3)$ ($i=1,2$);
\item $(1,3,3,0,0)$, for $\alpha=\alpha_1$, $\gamma=\alpha_2$ and $i=3$ in type $B_3$;
\item $(2,3,2,0,0)$, for $\alpha=\alpha_2$, $\gamma=\alpha_2+\alpha_3$ and $i=1$ in type $C_3$;
\item $(2,2,3,0,1)$, for $\alpha=\alpha_1$, $\gamma=\alpha_3$ and $i=2$ in type $C_3$.
\end{itemize}

Now we prove \Cref{claim:bracket-not-zero} for each possible $5$-tuple. For $(1,t,t,0,0)$, $t\in\bI_3$,
\begin{align*}
M_{\alpha}&=\Bbbk e_{\alpha}\simeq L_1, & M_{\beta}&=\bigoplus_{s=0}^{t-1} \Bbbk e_{\beta+s\alpha_i}\simeq L_t,  & M_{\gamma}&=\bigoplus_{s=0}^{t-1} \Bbbk e_{\gamma+s\alpha_i}\simeq L_t.
\end{align*}
By \Cref{prop:generation-degree1-alphap} there exists $\mathtt{c}\ne0$ such that $[e_{\alpha},e_{\gamma}]=\mathtt{c} e_{\beta}$.
As $[e_i,e_{\alpha}]=0$,  we have that 
\begin{align*}
[e_{\alpha},e_{\gamma+s\alpha_i}]=[e_{\alpha},(\ad e_i)^se_{\gamma}]=(\ad e_i)^s[e_{\alpha},e_{\gamma}]=\mathtt{c} e_{\beta+s\alpha_i},
\end{align*}
which implies that $[-,-]:\ov{e_{\alpha'}}\ot\ov{e_{\gamma'}}=\ttL_1\ot\ttL_t \to\ov{e_{\beta'}}=\ttL_t$ is $\mathtt{c}$ times the canonical map.

For $(2,1,2,0,1)$, we have that $\beta=\alpha+\gamma+\alpha_i$ and
\begin{align*}
M_{\alpha}&=\Bbbk e_{\alpha}\oplus \Bbbk (\ad e_i)e_{\alpha} \simeq L_2, & M_{\beta}&= \Bbbk e_{\beta}\simeq L_1,  & 
M_{\gamma}&=\Bbbk e_{\gamma}\oplus \Bbbk (\ad e_i)e_{\gamma} \simeq L_2.
\end{align*}
By \Cref{prop:generation-degree1-alphap}, there exists $\mathtt{c}\ne0$ such that $[e_{\alpha},(\ad e_i)e_{\gamma}]=\mathtt{c} e_{\beta}$.
As the $i$-string of $\beta$ has length 1, $\beta\pm\alpha_i\notin\varDelta_{+}$, so
$[e_i,e_{\beta}]=[e_{\alpha},e_{\gamma}]=0$. Thus 
\begin{align*}
[(\ad e_i)e_{\alpha},e_{\gamma}]=-[e_{\alpha},(\ad e_i)e_{\gamma}]=-\mathtt{c} e_{\beta},
\end{align*}
which implies that $[-,-]:\ov{e_{\alpha'}}\ot\ov{e_{\gamma'}}=\ttL_2\ot\ttL_2 \to\ov{e_{\beta'}}=\ttL_1$ is $2\mathtt{c}$ times the canonical map, see Example \ref{ex:alphap-isotypic-L2}.

For $(2,3,2,0,0)$, $p>3$, we have that $\beta=\alpha+\gamma$ and
\begin{align*}
M_{\alpha}&=\Bbbk e_{\alpha}\oplus \Bbbk (\ad e_i)e_{\alpha} \simeq L_2, & 
M_{\beta}&= \bigoplus_{0\le s\le2} \Bbbk (\ad e_i)^se_{\beta}\simeq L_3,  & 
M_{\gamma}&=\Bbbk e_{\gamma}\oplus \Bbbk (\ad e_i)e_{\gamma} \simeq L_2.
\end{align*}
By \Cref{prop:generation-degree1-alphap}, there exists $\mathtt{c}\ne0$ such that $[e_{\alpha},e_{\gamma}]=\mathtt{c} e_{\beta}$.
The copy of $L_3$ inside $M_{\alpha}\ot M_{\gamma}$ as in Example \ref{ex:alphap-isotypic-L2} is spanned by 
$e_{\alpha} \ot e_{\gamma}$, $e_{\alpha}\ot (\ad e_i)e_{\gamma}+(\ad e_i)e_{\alpha}\ot e_{\gamma}$ and $2(\ad e_i)e_{\alpha}\ot (\ad e_i)e_{\gamma}$. Using the Jacobi identity and that $(\ad e_i)^2e_{\alpha}=(\ad e_i)^2e_{\gamma}=0$,
\begin{align*}
\mathtt{c} (\ad e_i)e_{\beta} &= (\ad e_i)[e_{\alpha},e_{\gamma}]=[e_{\alpha},(\ad e_i)e_{\gamma}]+[(\ad e_i)e_{\alpha}, e_{\gamma}],
\\
\mathtt{c} (\ad e_i)^2 e_{\beta} &= (\ad e_i)^2[e_{\alpha},e_{\gamma}]=
2[(\ad e_i)e_{\alpha},(\ad e_i)e_{\gamma}],
\end{align*}
so $[-,-]:\ov{e_{\alpha'}}\ot\ov{e_{\gamma'}}=\ttL_2\ot\ttL_2 \to\ov{e_{\beta'}}=\ttL_3$ is $\mathtt{c}$ times the canonical map.

For $(2,2,3,0,1)$, $p>3$, we have that $\beta=\alpha+\gamma+\alpha_i$ and
\begin{align*}
M_{\alpha}&=\Bbbk e_{\alpha}\oplus \Bbbk (\ad e_i)e_{\alpha} \simeq L_2, & 
M_{\beta}&= \Bbbk e_{\beta}\oplus \Bbbk (\ad e_i)e_{\beta} \simeq L_2,  & 
M_{\gamma}&=\bigoplus_{0\le s\le2} \Bbbk (\ad e_i)^se_{\gamma}\simeq L_3.
\end{align*}
By \Cref{prop:generation-degree1-alphap}, there exists $\mathtt{c}\ne0$ such that $[e_{\alpha},(\ad e_i)e_{\gamma}]=\mathtt{c} e_{\beta}$.
The copy of $L_2$ inside $M_{\alpha}\ot M_{\gamma}$ as in Example \ref{ex:alphap-isotypic-L3-L2-p>3} is spanned by 
$e_{\alpha}\ot (\ad e_i)e_{\gamma}-2(\ad e_i)e_{\alpha}\ot e_{\gamma}$ and $e_{\alpha}\ot (\ad e_i)^2e_{\gamma}-(\ad e_i)e_{\alpha}\ot (\ad e_i)e_{\gamma}$. Using the Jacobi identity and that $[e_{\alpha},e_{\gamma}]=0$, we check that $[-,-]:\ov{e_{\alpha'}}\ot\ov{e_{\gamma'}}=\ttL_2\ot\ttL_3 \to\ov{e_{\beta'}}=\ttL_2$ is $3\mathtt{c}$ times the canonical map.

\medspace

Finally \ref{item:delta-parabolic-root-system} follows from \Cref{lem:Cuntz-Lentner} and the following claim: if $\beta\in \bZ^{\theta-1}$ has coprime entries and $m\ge 2$ is such that $m\beta\in\pi_i(\nabla^A)$, then $\beta\in\pi_i(\nabla^A)$. The proof of the claim follows from  \Cref{prop:nabla-delta-parabolic-root-system} and \Crefrange{ex:A2}{ex:br(3)}.
\end{proof}

Next we describe explicitly some examples of Lie algebras in $\Ver_p$ obtained by semisimplification as in \Cref{thm:ss-description-S(g)}. To describe them as objects in $\Ver_p$, we use the notation $\ttL_i^{(n)}$: a copy of $\ttL_i$ in degree $n\in\bZ$. Also, $\ttb:\sfS(\fg, e_i)\ot \sfS(\fg, e_i)\to\sfS(\fg, e_i)$ denotes the bracket.

\begin{example}\label{ex:A2-explicit-structure}
As in \Cref{ex:A2,ex:B2}, we take $\fg$ of type either $A_2$ or $B_2$, and consider the semisimplification with respect to $e_1$. The Cartan matrix is $\begin{pmatrix}
2 & -a \\ -1 & 2 \end{pmatrix}$, with $a$ being either $1$ for type $A_2$, or $a=2$ for $B_2$. As an object in $\Ver_p$, we have
\begin{align*}
\sfS(\fg, e_1)&=\ttL_{a+1}^{(1)}\oplus \left(\ttL_3^{(0)} \oplus \ov \ttL_1^{(0)} \right)\oplus \ttL_{a+1}^{(-1)}, & 
\ttL_{a+1}^{(1)}&=\overline{M_2}, & \ttL_3^{(0)}&=\overline{S}, & \ttL_1^{(0)}&=\overline{h_2}, & \ttL_{a+1}^{(-1)}&=\overline{N_2}.
\end{align*}
Here, $\ttL_3^{(n)}=0$ if $p=3$. By \Cref{lem:torus-action-alphap,lem:bracket-sl2-g-alphap,lem:bracket-same-alphap},
\begin{align*}
\restr{\ttb}{\ttL_1^{(0)}\ot\ttL_{a+1}^{(\pm 1)}}&=\pm (4-a)\iota_{\ttL_{a+1}^{(\pm 1)}}, &
\restr{\ttb}{\ttL_{a+1}^{(1)}\ot\ttL_{a+1}^{(-1)}}&=6^{a-1}\iota_{\ttL_1^{(0)}}\oplus(-2^{2a-1}\iota_{\ttL_3^{(0)}}),
\\
\restr{\ttb}{\ttL_3^{(0)}\ot\ttL_{a+1}^{(\pm 1)}}&=(a+2)\iota_{\ttL_{a+1}^{(\pm 1)}}, &
\restr{\ttb}{\ttL_1^{(0)}\ot\ttL_3^{(0)}}&=\restr{\ttb}{\ttL_{a+1}^{(\pm 1)}\ot\ttL_{a+1}^{(\pm 1)}}=0.
\end{align*}
\end{example}

\begin{example}\label{ex:B2-explicit-structure}
Let $\fg$ be of type $B_2$ as in \Cref{ex:B2}, $i=2$. As an object in $\Ver_p$, 
\begin{align*}
\sfS(\fg, e_2) &=\ttL_1^{(2)}\oplus \ttL_2^{(1)}\oplus \left(\ttL_3^{(0)} \oplus \ov \ttL_1^{(0)} \right) \oplus \ttL_2^{(-1)}\oplus \ttL_1^{(-2)}.
\end{align*}
By \Cref{lem:torus-action-alphap,lem:bracket-sl2-g-alphap,lem:bracket-same-alphap} and direct computation,
\begin{align*}
\restr{\ttb}{\ttL_1^{(0)}\ot\ttL_2^{(\pm 1)}}&=\pm 2\iota_{\ttL_2^{(\pm 1)}}, &
\restr{\ttb}{\ttL_3^{(0)}\ot\ttL_2^{(\pm 1)}}&=2\iota_{\ttL_2^{(\pm 1)}}, &
\restr{\ttb}{\ttL_2^{(1)}\ot\ttL_2^{(-1)}}&=\iota_{\ttL_1^{(0)}}\oplus(-2\iota_{\ttL_3^{(0)}}),
\\
\restr{\ttb}{\ttL_1^{(0)}\ot\ttL_1^{(\pm 2)}}&=\pm 4\iota_{\ttL_1^{(\pm 2)}},  &
\restr{\ttb}{\ttL_2^{(\pm 1)}\ot\ttL_2^{(\pm 1)}}&=2\iota_{\ttL_1^{(\pm 2)}},&
\restr{\ttb}{\ttL_1^{(2)}\ot\ttL_1^{(-2)}}&=\iota_{\ttL_1^{(0)}},
\end{align*}
and the remaining brackets $\restr{\ttb}{\ttL_i^{(m)}\ot\ttL_j^{(\pm n)}}$ are zero.
\end{example}

\begin{example}\label{ex:G2-explicit-structure}
Assume $p>3$ and consider $\fg$ of type $G_2$ as in \Cref{ex:G2}. For $i=1$,
\begin{align*}
\sfS(\fg, e_1) &= \ttL_1^{(2)}\oplus \ttL_4^{(1)}\oplus \left(\ttL_3^{(0)} \oplus \ov \ttL_1^{(0)} \right) \oplus \ttL_4^{(-1)}\oplus \ttL_1^{(-2)}.
\end{align*}
By \Cref{lem:torus-action-alphap,lem:bracket-sl2-g-alphap,lem:bracket-same-alphap} and direct computation,
\begin{align*}
\restr{\ttb}{\ttL_1^{(0)}\ot\ttL_4^{(\pm 1)}}&=\pm \iota_{\ttL_2^{(\pm 1)}}, &
\restr{\ttb}{\ttL_3^{(0)}\ot\ttL_4^{(\pm 1)}}&=15\iota_{\ttL_4^{(\pm 1)}}, &
\restr{\ttb}{\ttL_4^{(1)}\ot\ttL_4^{(-1)}}&=72\iota_{\ttL_1^{(0)}}\oplus(-120\iota_{\ttL_3^{(0)}}),
\\
\restr{\ttb}{\ttL_1^{(0)}\ot\ttL_1^{(\pm 2)}}&=\pm 2\iota_{\ttL_1^{(\pm 2)}},  &
\restr{\ttb}{\ttL_4^{(\pm 1)}\ot\ttL_4^{(\pm 1)}}&=4\iota_{\ttL_1^{(\pm 2)}},&
\restr{\ttb}{\ttL_1^{(2)}\ot\ttL_1^{(-2)}}&=3\iota_{\ttL_1^{(0)}},
\\
\restr{\ttb}{\ttL_1^{(2)}\ot\ttL_4^{(-1)}}&= \iota_{\ttL_4^{(1)}},  &
\restr{\ttb}{\ttL_1^{(-2)}\ot\ttL_4^{(1)}}&= \iota_{\ttL_4^{(-1)}},  &
\end{align*}
and the remaining brackets $\restr{\ttb}{\ttL_i^{(m)}\ot\ttL_j^{(\pm n)}}$ are zero.

Set now $i=2$. As an object in $\Ver_p$,
\begin{align*}
\sfS(\fg, e_2) &= \ttL_2^{(3)}\oplus \ttL_1^{(2)}\oplus \ttL_2^{(1)}\oplus \left(\ttL_3^{(0)} \oplus \ov \ttL_1^{(0)} \right) \oplus \ttL_2^{(-1)}\oplus \ttL_1^{(-2)} \oplus \ttL_2^{(-3)}.
\end{align*}
By \Cref{lem:torus-action-alphap,lem:bracket-sl2-g-alphap,lem:bracket-same-alphap} and direct computation,
\begin{align*}
\restr{\ttb}{\ttL_1^{(0)}\ot\ttL_2^{(\pm 1)}}&=\pm \iota_{\ttL_2^{(\pm 1)}}, &
\restr{\ttb}{\ttL_3^{(0)}\ot\ttL_2^{(\pm 1)}}&=3\iota_{\ttL_2^{(\pm 1)}}, &
\restr{\ttb}{\ttL_2^{(1)}\ot\ttL_2^{(-1)}}&=\iota_{\ttL_1^{(0)}}\oplus(-3\iota_{\ttL_3^{(0)}}),
\\
\restr{\ttb}{\ttL_1^{(0)}\ot\ttL_1^{(\pm 2)}}&=\pm 2\iota_{\ttL_1^{(\pm 2)}}, &
\restr{\ttb}{\ttL_3^{(0)}\ot\ttL_2^{(\pm 3)}}&=3\iota_{\ttL_2^{(\pm 3)}}, &
\restr{\ttb}{\ttL_1^{(2)}\ot\ttL_1^{(-2)}}&=3\iota_{\ttL_1^{(0)}},
\\
\restr{\ttb}{\ttL_1^{(0)}\ot\ttL_2^{(\pm 3)}}&=\pm 3\iota_{\ttL_2^{(\pm 3)}}, &
\restr{\ttb}{\ttL_2^{(\pm 1)}\ot\ttL_2^{(\pm 1)}}&=\iota_{\ttL_1^{(\pm 2)}}, &
\restr{\ttb}{\ttL_2^{(3)}\ot\ttL_2^{(-3)}}&=6\iota_{\ttL_1^{(0)}}\oplus(-6\iota_{\ttL_3^{(0)}}),
\\
\restr{\ttb}{\ttL_2^{(\pm 1)}\ot\ttL_1^{(\pm 2)}}&=\iota_{\ttL_2^{(\pm 3)}}, &
\restr{\ttb}{\ttL_1^{(\pm 2)}\ot\ttL_2^{(\mp 1)}}&=\iota_{\ttL_2^{(\pm 1)}},
\\
\restr{\ttb}{\ttL_2^{(\pm 3)}\ot\ttL_2^{(\mp 1)}}&=3\iota_{\ttL_1^{(\pm 2)}}, &
\restr{\ttb}{\ttL_2^{(\pm 3)}\ot\ttL_1^{(\mp 2)}}&=6\iota_{\ttL_2^{(\pm 1)}},
\end{align*}
and the remaining brackets $\restr{\ttb}{\ttL_i^{(m)}\ot\ttL_j^{(\pm n)}}$ are zero.
\end{example}

\begin{corollary}\label{coro:ss-description-S(g)-p=3}
Let $p=3$, and take $A$ and $i$ as in \Cref{thm:ss-description-S(g)}\ref{item:gen-deg-one}. Then $\sfS(\fg, e_i) \simeq \fg(B,\bp)$, where
\begin{itemize}
\item $\bp=(p_j)_{j\in\ov{\bI}}$, $p_j=\begin{cases} 1, & a_{ij}=0, \\ -1, & a_{ij}\ne 0;\end{cases}$
\item $B=(b_{jk})_{j,k\in\ov{\bI}}$, $b_{jk}=\begin{cases} a_{jk}, & a_{ij}=0, \\ -a_{jk}-a_{ji}a_{ik}, & a_{ij}\ne 0.\end{cases}$
\end{itemize}
\end{corollary}
\begin{proof}
Let $\ov{\fg}:=\sfS(\fg, e_i)$.
The $\bZ^{\theta-1}$-grading of $\ov{\fg}$ induces a $\bZ$-grading such that $\ov{e_j}\in \ov{\fg}_1$, $\ov{f_j}\in \ov{\fg}_{-1}$, $j\in\ov{\bI}$, and $\ov{h_k}\in\ov{\fg}_0$, $k\ne i$. By Proposition \ref{prop:semisimplification-multidimension} \ref{item:semisimplification-multidimension-i}, if $a_{ij}=0$, then $M_{\alpha_j}\simeq L_1\simeq N_{\alpha_j}$, so $\ov{e_j}$, $\ov{f_j}$ are even, while if $a_{ij}\ne0$ and $j\in\ov{\bI}$, then $M_{\alpha_j}\simeq L_2\simeq N_{\alpha_j}$; thus the parity of the $\ov{e_j}$'s and the $\ov{f_j}$'s is given by $\bp$.

By Theorem \ref{thm:ss-description-S(g)}, $\ov{\fg}$ is generated by $\ov{e_j}$, $\ov{f_j}$, $j\in\ov{\bI}$, and $\ov{h_k}$, $k\ne i$, since $S=0$.
By Lemmas \ref{lem:bracket-different-alphap} and \ref{lem:bracket-same-alphap}, $[\ov{e_j},\ov{f_k}]=\delta_{jk}\ov{h_j}$ up to rescale the $\ov{f_j}$'s. 
Following Notation \ref{notation:indecomposables-bases},
\begin{align*}
[\widetilde{h_{j}}, e_k]=\begin{cases}
a_{jk}e_k, & a_{ij}=0, \\ -(a_{jk}+a_{ji}a_{ik})e_k, & a_{ij}\ne 0.
\end{cases}
\end{align*}
Thus $[\ov{h_j},\ov{e_k}]=b_{jk} \ov{e_k}$ for all $j,k\in\ov{\bI}$. Analogously, $[\ov{h_j},\ov{f_k}]=-b_{jk} \ov{f_k}$ for all $j,k\in\ov{\bI}$. 
As $\fh$ is abelian, we also have that $[\ov{h_j},\ov{h_k}]=0$ for all $j,k$.  Thus relations \eqref{eq:relaciones gtilde} hold in $\ov{\fg}$, which implies that there is a surjective map $\widetilde{\fg}(B,\bp)\twoheadrightarrow \ov{\fg}$. 

As $\fg(B,\bp)$ is the quotient of $\widetilde{\fg}(B,\bp)$ by the maximal ideal trivially intersecting $\fh$, the map above factorizes through a surjective map $\pi:\ov{\fg}\twoheadrightarrow \fg(B,\bp)$. As $\ov{\fg}$ has a non-degenerate invariant symmetric bilinear form by Theorem \ref{thm:ss-description-S(g)}, the map $\pi$ is an isomorphism.
\end{proof}

\begin{remark}
This corollary is related with \cite{Kan}*{Theorem 3.6.5} when the characteristic is three. It applies to those examples in loc. cit. where the element of the Lie algebra is homogeneous.
\end{remark}

\begin{example}
We consider semisimplifications of the Lie algebra $\brown(3)$, see Example \ref{ex:br(3)}.
\begin{enumerate}[leftmargin=*]
\item The semisimplification of $\fg(A_1)$ under $\ad e_1$ is the Lie superalgebra $\fg(B_1,\bp_1)$, where
\begin{align*}
B_1 &=\begin{bmatrix} 0 & 1 \\ 1 & 0 \end{bmatrix}, & \bp_1 =(1,0).
\end{align*}

\item The semisimplification of $\fg(A_1)$ under $\ad e_2$ is the Lie superalgebra $\fg(B_2,\bp_2)$, where
\begin{align*}
B_2 &=\begin{bmatrix} 0 & 1 \\ -1 & 2 \end{bmatrix}, & \bp_2 =(1,1).
\end{align*}

\item The semisimplification of $\fg(A_2)$ under $\ad e_1$ is the Lie superalgebra $\fg(B_3,\bp_3)$, where
\begin{align*}
B_3 &=\begin{bmatrix} 2 & -2 \\ 1 & 0 \end{bmatrix}, & \bp_3=(1,0).
\end{align*}

\item The semisimplification of $\fg(A_2)$ under $\ad e_2$ does not fit in the context of \Cref{coro:ss-description-S(g)-p=3} since $\beta=\alpha_1+\alpha_2+\alpha_3$ is not $2$-good.
\end{enumerate}
That is, the three possible semisimplifications give the Lie superalgebra $\sbrown(2;3)$: we recover the three possible realizations of $\sbrown(2;3)$ as a contragredient Lie superalgebra. This corresponds to the construction given in \cite{Kan}*{\S 4.2}.
\end{example}

\subsection*{Acknowledgements} 
We are deeply grateful to Pavel Etingof and Arun Kannan for many helpful discussions. 
I.A. and G.S were partially supported by the Alexander von Humboldt Foundation through a Linkage Programme.
I. A. was partially supported by Conicet, SeCyT (UNC) and MinCyT. G.S was partially supported by an NSF Simons travel grant. 
The research of J.P. was partially supported by NSF grant DMS-2146392 and by Simons Foundation Award 889000 as part of the Simons Collaboration on Global Categorical Symmetries. J.P. would like to thank the hospitality
and excellent working conditions at the Department of Mathematics at the University of Hamburg, where J. P. has carried out part of this research as an Experienced Fellow of the Alexander von Humboldt Foundation.
The three authors would like to thank I. Heckenberger for the kind invitation and hospitality during the Workshop Hopf algebras and Tensor Categories in Marburg in May 2023. I.A. and G. S. would also thank J. P. and the people at the University of Hamburg for the excellent working condition during a short visit in May 2023, where big contributions to this work were carried out.

\bibliography{biblio}

\end{document}